\documentclass[reqno, 11pt]{amsart}
\usepackage[left=2cm, right=2cm, top=2cm, bottom=2cm]{geometry}
\usepackage[T1]{fontenc}
\usepackage[utf8]{inputenc}	
\usepackage[english]{babel}
\usepackage[babel]{csquotes}
\usepackage{amssymb}
\usepackage{amsmath}
\usepackage{amsthm}
\usepackage{latexsym}
\usepackage{xcolor}
\usepackage{mathrsfs}
\usepackage{bbm}
\usepackage{verbatim}
\usepackage{graphicx}
\usepackage{array}
\usepackage{tikz}
\usepackage{subcaption}
\usepackage{tabularx}
\usepackage{pgfplots}
\usepackage{float}
\usepgfplotslibrary{fillbetween}

\newcolumntype{M}[1]{>{\centering\arraybackslash}m{#1}}
\usepackage[colorlinks=true, pdfstartview=FitV, linkcolor=blue, citecolor=blue, urlcolor=blue]{hyperref}
\usepackage[style=numeric, sorting=nty, sortcites=true, backend=bibtex, maxbibnames=10, giveninits=true]{biblatex}
\usepackage{faktor}
\usepackage[shortlabels]{enumitem}
\usepackage[ruled,vlined]{algorithm2e}

\definecolor{myblue}{RGB}{31,119,180}
\definecolor{myorange}{RGB}{255,127,14}
\definecolor{mygreen}{RGB}{44,160,44}
\definecolor{myred}{RGB}{214,39,40}
\definecolor{mypurple}{RGB}{148,103,189}
\definecolor{mybrown}{RGB}{140,86,75}

\numberwithin{equation}{section}
\newtheorem{thm}{Theorem}[section]

\newtheorem{lem}[thm]{Lemma}

\theoremstyle{definition}
\newtheorem{defin}[thm]{Definition}
\newtheorem{remark}[thm]{Remark}

\renewcommand{\d}{{\mathrm d}} 
\newcommand{\norm}[1]{\left\|#1\right\|} 
\newcommand{\ip}[2]{\left\langle #1,#2 \right\rangle} 

\def\supp{\sup_{\tau \in [0,t]}}

\def\uone{u_{\lambda_1}}
\def\utwo{u_{\lambda_2}}

\def\bn{{\boldsymbol n}}

\def\A{\mathcal A}
\def\bH{{\boldsymbol H}}

\def\b #1{{\boldsymbol #1}}
\def\enne{\mathbb{N}}

\def\erre{\mathbb{R}}

\def\P{\mathbb{P}}

\def\E{\mathop{{}\mathbb{E}}}

\def\cL{\mathscr{L}}
\def\cF{\mathscr{F}}
\def\cB{\mathscr{B}}
\def\eps{\varepsilon}
\def\cP{\mathscr{P}}
\def\cZ{\mathscr{Z}}
\def\OO{\mathcal{O}}
\def\embed{\hookrightarrow}
\makeatletter
\DeclareFontFamily{OMX}{MnSymbolE}{}
\DeclareSymbolFont{MnLargeSymbols}{OMX}{MnSymbolE}{m}{n}
\SetSymbolFont{MnLargeSymbols}{bold}{OMX}{MnSymbolE}{b}{n}
\DeclareFontShape{OMX}{MnSymbolE}{m}{n}{
	<-6>  MnSymbolE5
	<6-7>  MnSymbolE6
	<7-8>  MnSymbolE7
	<8-9>  MnSymbolE8
	<9-10> MnSymbolE9
	<10-12> MnSymbolE10
	<12->   MnSymbolE12
}{}
\DeclareFontShape{OMX}{MnSymbolE}{b}{n}{
	<-6>  MnSymbolE-Bold5
	<6-7>  MnSymbolE-Bold6
	<7-8>  MnSymbolE-Bold7
	<8-9>  MnSymbolE-Bold8
	<9-10> MnSymbolE-Bold9
	<10-12> MnSymbolE-Bold10
	<12->   MnSymbolE-Bold12
}{}

\let\llangle\@undefined
\let\rrangle\@undefined
\DeclareMathDelimiter{\llangle}{\mathopen}%
{MnLargeSymbols}{'164}{MnLargeSymbols}{'164}
\DeclareMathDelimiter{\rrangle}{\mathclose}%
{MnLargeSymbols}{'171}{MnLargeSymbols}{'171}
\makeatother
\AtEveryBibitem{%
	\clearfield{doi}%
	\clearfield{url}%
	\clearfield{issn}%
	\clearfield{isbn}
} 
\renewbibmacro*{publisher+location+date}{%
	\ifentrytype{book}
	{\printlist{publisher}%
		\iflistundef{location}
		{\setunit*{\addcomma\space}}
		{\setunit*{\addcomma\space}}%
		\printlist{location}}
	{\printlist{location}%
		\iflistundef{publisher}
		{\setunit*{\addcomma\space}}
		{\setunit*{\addcomma\space}}%
		\printlist{publisher}}
	\setunit*{\addcomma\space}%
	\usebibmacro{date}%
	\newunit} 

\def\andrea#1{{\color{black}#1}}
\def\marghe#1{{\color{black}#1}}
\def\luca#1{{\color{black}#1}}

\begin{document}	
	\title[An AC equation with jump-diffusion noise for biological damage and repair processes]{An Allen--Cahn equation with jump-diffusion noise for biological damage and repair processes}
	
	\author{Andrea Di Primio \and Marvin Fritz \and Luca Scarpa \and Margherita Zanella}
	\address{Classe di Scienze, Scuola Normale Superiore, Piazza dei Cavalieri 7, 56126 Pisa, Italy}
	\email{andrea.diprimio@sns.it}
    \address{Johann Radon Institute for Computational and Applied Mathematics, Altenberger Str. 69, 4040 Linz, Austria} \email{marvin.fritz@ricam.oeaw.ac.at}
    \address{Dipartimento di Matematica,
		Politecnico di Milano, Via E.~Bonardi 9, 20133 Milano, Italy}
	\email{luca.scarpa@polimi.it}
    \address{Dipartimento di Matematica,
		Politecnico di Milano, Via E.~Bonardi 9, 20133 Milano, Italy}
	\email{margherita.zanella@polimi.it}
	
	\begin{abstract}
		This paper analyzes a stochastic Allen--Cahn equation for the dynamics of biomolecular damage and repair.
        The system is driven by two distinct noise processes: a multiplicative cylindrical Wiener process, modeling continuous background stochastic fluctuations, and a jump-type noise, modeling the abrupt, localized damage 
        induced by external shocks. The drift of the equation 
        is singular and covers the typical logarithmic 
        Flory-Huggins potential required in phase-separation dynamics.
        We prove well-posedness of the model in a strong 
        probabilistic sense, and analyze its long-time behavior 
        in terms of existence and uniqueness of invariant measures, ergodicity, and mixing properties.
        Eventually, we present an Euler--Maruyama scheme to simulate the model and illustrate how it captures fundamental biological phenomena, 
        such as damage clustering, stress-induced 
        topology perturbations, and damage dynamics.        
	\end{abstract}
	\maketitle
	
\section{Introduction}
The integrity of DNA and cellular tissues is constantly challenged by internal and external stressors.
Understanding how damage arises and how repair proceeds is a central problem in biophysics and medicine. 
Many relevant incidents are neither smooth nor continuous: double-strand breaks, replication errors, chemically induced lesions, or other localized failures can occur abruptly and unpredictably, and may cluster in space. 
At the same time, background variability (metabolic noise, reactive oxygen species, heterogeneous repair activity) acts persistently and on smaller scales. 
A realistic mesoscopic model should therefore combine (i) bistable relaxation between ``healthy'' and ``damaged'' states, (ii) diffusion-like repair cooperation, and (iii) the coexistence of continuous fluctuations with discrete, catastrophic damage hits. The Allen--Cahn equation, originally developed to model phase separation (see \cite{allen1979}), provides a natural framework for describing such bistable systems. Its deterministic dynamics captures the evolution towards one of two stable states (namely, “healthy” or “damaged”) with the diffusion term representing spatial repair interactions. 
Stochastic variants of the Allen–Cahn equation typically 
take into account only
continuous perturbations by Gaussian noise.
However, accurately modeling biological damage requires incorporating also discrete, abrupt shocks. In this work, we introduce and analyze a stochastic Allen--Cahn model with both multiplicative Wiener and jump noise terms.
The model reads
\[
		\begin{cases}
			\d u - \varepsilon \Delta u  \: \d t
            +\Phi'(u)\:\d t = 
            G(u) \: \d W + \displaystyle \int_Z J(u^-,z) \: \overline{\mu}(\d t,\, \d z) & \quad \text{in }  (0,T) \times\OO, \\
			\alpha_d u + \alpha_n \partial_{\bn} u = 0 & \quad \text{on } (0,T) \times \partial \OO, \\
			u(\cdot, 0) = u_0 & \quad \text{in } \OO,
		\end{cases}
\]
    where $u(t,x)$ denotes the local damage level at position $x\in\OO$ and time $t\ge 0$, and 
    the coefficients $\alpha_d,\alpha_n\in\{0,1\}$ with $\alpha_d+\alpha_n=1$ allow either Dirichlet or Neumann homogeneous boundary conditions.
The classical diffuse-interface perspective of the Allen-Cahn 
modeling requires $u$ to be confined in the 
physically-relevant interval $[0,1]$, with $0$ and $1$
representing the so-called pure phases, with the convention
\[
u=0 \quad \text{(healthy)} , 
\qquad 
u=1 \quad \text{(fully damaged)} .
\]
Instead, intermediate values of $u$ 
describe sublethal, potentially repairable states: 
this is consistent with evidence that damage accumulates gradually rather than in an all-or-nothing manner \cite{curtis1986}. 
Mathematically, this is ensured by a suitable choice of the potential $\Phi$, which enforces barriers at the endpoints. 
The typical choice for $\Phi$ required in 
thermodynamics is
the well-know double-well Flory-Huggins potential \cite{Flory42, Huggins41}
\begin{equation}
\label{FH}
    \Phi(u)=\theta\Big(u\ln u+(1-u)\ln(1-u)\Big)
    -\theta_0\left(u-\dfrac 12\right)^2, \qquad u\in[0,1].
\end{equation}
The Allen-Cahn dynamics results then as a gradient flow 
with respect to the $L^2$ metric of the free-energy 
\[
  u\mapsto \frac\eps2\int_\OO|\nabla u(x)|^2\,\d x + 
  \int_\OO\Phi(u(x))\,\d x.
\]
where $\varepsilon>0$ controls the thickness of interfacial layers separating healthy and damaged regions.
From the modeling point of view, the diffuse-interface dynamics
capture recovery and cell-fate transitions, 
diffusion-driven repair cooperation, and constrained bistable dynamics. 
In our analysis, the model
will be allowed to be even more general, in the form 
\begin{equation} \label{eq:strongform}
		\begin{cases}
			\d u - \varepsilon \Delta u  \: \d t
            +\partial\Psi(u)\:\d t +F(u)\:\d t \ni
            G(u) \: \d W + \displaystyle \int_Z J(u^-,z) \: \overline{\mu}(\d t,\, \d z) & \quad \text{in }  (0,T) \times\OO, \\
			\alpha_d u + \alpha_n \partial_{\bn} u = 0 & \quad \text{on } (0,T) \times \partial \OO, \\
			u(\cdot, 0) = u_0 & \quad \text{in } \OO.
		\end{cases}
	\end{equation}
The derivative of the double-well potential
is replaced by the multi-valued sub-differential operator 
$\partial\Psi+F$, where $F$ is a Lipschitz-continuous function
and
$\Psi$ is a proper, convex, lower semicontinuous function on $\mathbb R$
with $\operatorname{Int}D(\Psi)=(0,L)$, and $L\in(0,\+\infty]$.
The case $L<+\infty$ model double-barrier potentials while 
$L=+\infty$ corresponds to a single-barrier in $0$ only.
In particular, this generality allows 
to possibly include non-smooth energies and 
to model additional biochemical regulation, such as time-scale separation in repair pathways or environmental modulation.

The stochastic components encode two biologically distinct perturbations. 
The Wiener noise $G(u)\,\mathrm{d}W$ represents continuous background fluctuations, while
the jump term
$\int_Z J(u^-,z)\,\overline{\mu}(\mathrm{d}t,\mathrm{d}z)$,
driven by a compensated Poisson random measure $\bar \mu$, models abrupt localized damage hits occurring at random times and locations. 
In particular, we allow the 
Wiener noise intensity $G$ and the
jump amplitude $J$ to be state dependent through a pre-factor that vanishes at the potential barriers, reflecting the idea that fully healthy and fully damaged configurations are comparatively insensitive to small additional perturbations, whereas intermediate configurations are the most susceptible. 
By selecting the event statistics and spatial kernels appropriately, this framework can also be specialized to model ionizing radiation tracks as a concrete application, see \cite{cordoni2024spatial,kellerer1984,goodhead1994,hill2020radiation}.

From a mathematical perspective, the main challenge is to deal 
with the possible singularity of the potential $\Phi$ at the 
barriers $0$ and $L$, and to confine the process $u$ 
into the physically-relevant interval $[0,L]$.
The main idea is to suitable tune the noise intensities
in a multiplicative fashion, in such a way that 
the singular energy proliferation at the potential barriers 
is kept under control throughout the evolution.
As far as the Wiener noise is concerned, this 
has been proposed for the stochastic Allen-Cahn equation 
in \cite{Bertacco21, SZ}, and also for coupled
Allen-Cahn-Navier-Stokes systems in \cite{DPGS, DPSZ}.
For the jump contribution, the presence of a singular potential 
is a novel challenge, as the main literature 
on stochastic evolutions equations with jump-type noise
mainly deal with regular nonlinearities \cite{pes-zab, applebaum2009}. The main idea to overcome such problem
is to prescribe a tuning of the jump intensity 
as for the Wiener case: roughly speaking, this allows to 
control the total energy over time and not to ``jump''
outside the relevant domain $(0,L)$.

The first main result of the paper consists of well-posedness 
for \eqref{eq:strongform} in the strong probabilistic sense.
The intuitive idea is that the jump intensity 
close to the potential barriers compensate the 
blow-up of the potential.
In general, such energy-noise balance may not be preserved
under standard approximation techniques.
From the technical point of view, 
the strategy of the proof
combines a Yosida regularization
for the double-well potential with a
ad-hoc counter-correction on the jump noise, in order to 
preserve such energy-noise balance also at the approximate level.
This allows to extend the classical frameworks for SPDEs with locally monotone nonlinearities and jump noise \cite{brzez}.

The second main result of the paper 
concerns existence and uniqueness of invariant measures for the 
associated transition semigroup, 
as well as ergodic and mixing properties.
These are obtained by deducing refined energy estimates 
on the solution in order to cast our framework in
the typical long-time behavior theory for
stochastic evolution equations.

 Eventually, we demonstrate the model’s computational relevance through an implicit–explicit Euler–Maruyama finite element scheme.
 This simulates damage-repair dynamics and visualizes the roles of stochastic fluctuations and radiation tracks in the most interesting case
 of the logarithmic potential \eqref{FH}.
 By tuning the parameters governing noise and jumps, the simulations reproduce characteristic radiobiological responses such as localized damage clustering, partial recovery, and cumulative loss of integrity under sustained irradiation.
 Furthermore, we are also able to visualize 
 an abstract random separation property, that was
 proved for the model without jumps in \cite{BOS21, orr-scar-sep}. This reads, in its general form, 
 \[
\P\left(\sup_{(t,x)\in\mathbb R_+\times\overline\OO}
|u(t,x)|<1\right)=1
\]
and confirms also from a mathematical viewpoint the fact that 
equilibrium dynamics for \eqref{FH} are not 
attained at the pure phases.

In conclusion, the Allen-Cahn equation has been studied in the last decades 
under numerous point of views. For example, 
without aiming for generality, we refer to 
\cite{BGJK, BVZ, DPGS2,
hairer2012triviality, metz, orr-scar}
for contributions on the Allen-Cahn equation with 
Wiener-type noise, \cite{web1, web2}
on its sharp interface limit, 
\cite{NNTM, DNNTM} for coupled systems,
and \cite{BHR, itojump, H}
for stochastic evolution equations with 
jump-type noise, as well as the references therein.
Up to the authors' knowledge, the study of
stochastic models with jump-noise and 
singular potentials has been an open issue 
in the last years. This paper provides a 
first rigorous analytical treatment
in a biophysically-motivated setting.

\textbf{Plan of the paper.} \andrea{The contents of the work are organized as follows.} Section \ref{sec:preliminaries} is devoted to illustrating the preliminary material, notation, and the \andrea{necessary notions} to state the main results. Section \ref{sec:wellposed} contains the proof of existence of probabilistically-strong solutions to the system. Section \ref{sec:invariant} contains the investigation of the longtime behavior of the system, in terms of existence and uniqueness of invariant measures. \andrea{Section \ref{sec:numerics} delves into a numerical investigation of the system.} Finally, a collection of useful technical estimates are presented in Appendix \ref{App_A}.
    
\section{Preliminaries and main results} \label{sec:preliminaries}	\subsection{Mathematical preliminaries and notation}
In the following, we illustrate the notation employed throughout the whole work.
\subsubsection{Stochastic framework.}
Let $T > 0$ be an arbitrary final time, and let us set $(\Omega,\cF,(\cF_t)_{t\in[0,T]},\P)$ to be a filtered probability space satisfying the usual conditions, i.e., the filtration $(\cF_t)_{t\in[0,T]}$ is saturated and right-continuous. Let $\cP$ denote the progressive $\sigma$-algebra on $\Omega \times [0,T]$, accordingly. For every measurable space $(E,\, \mathscr E)$,
we denote by $\mathcal M(E)$ the set of positive measures on $(E, \mathscr E)$.
Oftentimes, the space $E$ is also equipped with the structure of Banach space: in these cases, we imply that the underlying $\sigma$-algebra $\mathscr E$ is the Borel $\sigma$-algebra $\mathscr B(E)$. Moreover, for any positive real quantity $p \geq 1$ the symbol $L^p(\Omega; E)$ denotes the set of strongly measurable $E$-valued random variables on $\Omega$ with finite moments up to order $p$. When the space $E$ is allowed to depend on time, as in classical Bochner spaces, we may denote the resulting space with $L^p_\cP(\Omega; E)$	to stress that measurability is intended with respect to the progressive 	$\sigma$-algebra $\cP$.	
\\
Let us now point out the precise interpretation of the stochastic perturbations appearing in \eqref{eq:strongform}. Concerning It\^{o} noise, we denote by $W$ a cylindrical Wiener process on some fixed separable Hilbert space $U$. Accordingly, we set $\{e_k\}_{k\in\enne} \subset U$ to be a fixed arbitrary orthonormal system for $U$. Let us recall that by the properties of cylindrical Wiener processes, the process $W$ admits the representation
		\begin{equation} \label{eq:representation}
			W = \sum_{k=0}^{+\infty} \beta_k e_k,
		\end{equation}
		where $\{\beta_k\}_{k \in \enne}$ is a family of real and independent Brownian motions. However, the series \eqref{eq:representation} can not be expected to always converge in $U$. A well known result establishes that, in general, convergence can be achieved in some larger Hilbert space $U_0$ such that $U \embed U_0$ with Hilbert--Schmidt embedding $\iota$. Moreover, it is possible to identify $W$ as a $Q^0$-Wiener process on $U_0$, for the trace-class operator $Q^0 := \iota \circ \iota^*$ on $U_0$ (see \cite[Subsection 2.5.1]{LiuRo}). In the following, we implicitly assume this extension by simply saying that $W$ is a cylindrical process on $U$. This holds also for stochastic integration à la It\^{o} with respect to $W$. Indeed, the symbol 
		\[
		\int_0^\cdot B(s)\,\d W(s) := \int_0^\cdot B(s) \circ \iota^{-1}\,\d W(s),
		\]
		for every process $B \in L^2_\cP(\Omega;L^2(0,T;\cL^2(U,K)))$, where $K$ is any real Hilbert space and $\cL^2(U,K)$
        is the space of Hilbert--Schmidt operators from $U$ to $K$ (a precise definition is given later on). It is well known that such a  definition is well posed and does not depend on the choice of $U_0$ or $\iota$ (see \cite[Subsection 2.5.2]{LiuRo}). 
        \\
        Let us now move to the jump noise term:
        \luca{we refer here to the classical theory \cite{jacod}.
        Let $(Z,\mathscr{Z})$ be a Blackwell space
        (e.g.~a Polish space with its Borel $\sigma$-algebra).
        A random measure is a map $\mu:\Omega\to\mathcal M(
        [0,+\infty)\times Z, \cB([0,+\infty))\otimes\cZ)$ such that
        $\mu(\omega, \{0\}\times Z)=0$ for all $\omega\in\Omega$. 
        A random measure $\mu$ is said to be integer-valued
        if:
        \begin{enumerate}
            \item $\mu(\omega, \{t\}\times Z)\leq1$ for every $\omega\in\Omega$ and $t\geq0$;
            \item $\mu(\omega, A)\in\mathbb N\cup\{+\infty\}$
            for every $\omega\in\Omega$ and $A\in\cB([0,+\infty))\otimes\cZ$;
            \item $\mu$ is optional and $\cB([0,+\infty))\otimes\cZ$-$\sigma$-finite.
        \end{enumerate}
        Furthermore, given a
        $\sigma$-finite positive measure $\nu$ on $(Z,\cZ)$,
        we say that $\mu$ is a time-homogeneous
        Poisson random measure with intensity $\nu$ if
        $\mu$ is an integer-valued random measure such that:
        \begin{enumerate}
            \item $\mu(\cdot, A)$ is a Poisson random 
            variable with rate $(\lambda\otimes\nu)(A)$
            for every $A\in \cB([0,+\infty))\otimes\cZ$,
            where $\lambda$ is the Lebesgue measure on 
            $[0,+\infty)$;
            \item $\{\mu(\cdot,A_i)\}_{i=1}^n$
            are independent for every disjoint $A_1,\ldots,A_n\in
            \cB([0,+\infty)\otimes\cZ$.
        \end{enumerate}
        The measure $\lambda\otimes\nu$ on $([0,+\infty)\times Z, \cB([0,+\infty)\otimes\cZ)$ 
        is called compensator of $\mu$, and $\bar \mu:=\mu-\lambda\otimes\nu$ is called compensated Poisson measure.
        The qualifier ``time-homogeneous'' underlines the fact that the compensator is precisely the product measure $\lambda \otimes \nu$. In general, this may not be the case.
        For every Hilbert space $K$, for every $T>0$, and for every 
        $f:\Omega\times[0,T]\times Z\to K$ that is 
        $\cP\otimes\cZ/\cB(K)$-measurable
        and with 
        $f\in L^1(\Omega; L^1(0,T; L^1(Z,\nu; K)))$, the stochastic process 
        \[
        (f\cdot\bar\mu)(t):=
        \int_{(0,t]}\int_Zf(s,z)\bar\mu(\d s, \d z), \quad t\in(0,T],
        \]
        is a well-defined $K$-valued martingale.
        Moreover, if also $f\in L^2(\Omega; L^2(0,T; L^2(Z,\nu; K)))$, then $f\cdot\bar\mu$ is a square-integrable
        martingale with predictable quadratic variation given by 
        \[
          \langle f\cdot\bar\mu\rangle(t)=\int_0^t
          \int_Z\|f(s,z)\|_K^2\,\nu(\d z)\,\d s, \quad t\in(0,T].
        \]
        The concept of Poisson random measure arises somewhat naturally in the theory of Poisson point processes. For convenience, we briefly recall the main ideas. Given a countable and locally finite subset $S \subset (0,+\infty)$, a point function on $(Z, \mathscr Z)$ is any mapping $p: S \to Z$. The counting measure associated to the point function $p$ is the measure $\mu_p$ on $((0,+\infty) \times Z,\mathscr B((0,+\infty)) \otimes \mathscr Z)$ defined as
        \[
        \mu_p((0,t] \times A) := \#\left\{s \in S \cap (0,t]: p(s) \in A\right\}, \quad t>0, \quad A\in \cZ.
        \]
         This measure counts how many events within $A$ have happened at most at time $t$. A point process is then a random variable with values in the set of point functions on $Z$, equipped with a suitable $\sigma$-algebra in such a way that the maps $p \mapsto \mu_p((0,t] \times A)$ are measurable for all $t > 0$ and $A \in \mathscr Z$. Then, given a point process $\Pi$, it may happen that the family of counting measures associated to $\Pi$ is a time-homogeneous Poisson random measure on $((0,+\infty) \times Z, \mathscr B((0,+\infty) \times \mathscr Z)$. In this case, we say that $\Pi$ is a time-homogeneous Poisson point process.
        }

    \subsubsection{Analytical framework.} Given any Banach space $E$, we denote with the symbol $\boldsymbol E$ the product space $E^d$ (or even $E^{d\times d}$, if no ambiguities arise), and by $E^*$ its topological dual. The corresponding duality pairing is denoted by $\ip{\cdot}{\cdot}_{E^*,E}$. Whenever $E$ is also a Hilbert space, then the scalar product of $E$ is denoted by $(\cdot,\cdot)_E$. For any pair of Banach spaces $E_1$ and $E_2$ and any extended real number $s \in [1,+\infty]$, the symbol $L^s(E_1; E_2)$ indicates the usual spaces of strongly measurable, Bochner-integrable functions
	defined on the Banach space $E_1$ and with values in the Banach spaces $E_2$. If $E_2$ is omitted, it is understood that $E_2 = \mathbb{R}$.
	If $E_1$ and $E_2$ are also Hilbert spaces, then we are able to define the space of Hilbert-Schmidt operators
	from $E_1$ to $E_2$ as $\cL^2(E_1,\,E_2)$, equipped with its usual norm structure given by $\norm{\,\cdot\,}_{\cL^2(E_1,\,E_2)}$.
	Within a smooth and bounded domain $\OO \subset \mathbb R^d$, with $d \in \{1,2,3\}$, we denote by $W^{s,p}(\OO)$ the classical Sobolev spaces of order $s \in \mathbb R$ and $p \in [1,+\infty]$, and we denote by $\norm{\,\cdot\,}_{W^{s,p}(\OO))}$ their classical norms (with the usual understanding that $W^{0,p}(\OO) = L^p(\OO)$). Accordingly, we define the Hilbert space $H^s(\OO):=W^{s,2}(\OO)$ for all $s\in\erre$,
	endowed with its canonical norm $\norm{\,\cdot\,}_{H^s(\OO)}$. In the following, it will be useful to introduce a specific notation for a variational structure, i.e., we set
	\[
	H:=L^2(\OO)\,, \qquad V:=
    \begin{cases}
        H^1_0(\OO) \quad&\text{if } (\alpha_d,\alpha_n)=(1,0),\\
        H^1(\OO) \quad&\text{if } (\alpha_d,\alpha_n)=(0,1),
    \end{cases} 
	\]
	endowed with their standard norms $\norm{\cdot}_H$ and
	$\norm{\cdot}_{V}$, respectively. As usual, we identify the Hilbert space $H$ with its dual through
	the corresponding Riesz isomorphism, so that we have the variational structure
	\[
	V\embed H \embed V^*,
	\]
	with dense and compact embeddings (both in the cases $d = 2$ and $d = 3$). In the Dirichlet case (i.e.~$(\alpha_d,\alpha_n)=(1,0)$), the zero-trace space
	\[
	H^1_0(\OO) = \{u \in H^1(\OO):u= 0 \text{ almost everywhere on } \partial \Omega\}
	\]
	is intended to be endowed with the $H^1$-seminorm structure, owing to the Poincaré inequality, namely 
	\[
	( u,\,v)_{H^1_0(\OO)} := (\nabla u, \,\nabla v)_{ \bH}, \qquad \|u\|_{H^1_0(\OO)} := \|\nabla u\|_ \bH, \qquad u,\,v \in V.
	\]
     The weak realization of the negative Dirichlet-Laplace operator, i.e., the operator
    \begin{align*}
        \mathcal{A}: H^1_0(\OO)\rightarrow H^{-1}(\OO),\quad \langle\mathcal{A}u,v\rangle_{H^{-1}(\OO), H^1_0(\OO)} := (\nabla u, \,\nabla v)_{ \bH}\quad \text{ for all }v\in H^1_0(\OO)
    \end{align*}
    is a linear isomorphism by the Riesz representation theorem. The spectral properties of $\mathcal A$ will be later useful. We denote by $\{\rho_n\}_{n \in \mathbb{N}}$ the countable sequence of eigenvalues of $\mathcal A$. In particular, it is an increasing sequence such that
\[
\lim_{n \rightarrow +\infty}\rho_n=+ \infty.
\]
Finally, we recall that $\rho_1^{-\frac 12}$ is the optimal constant in Poincar\'e's inequality 
\begin{equation}
\label{Poin}
\left\| u \right\|_H \leq \rho_1^{-\frac 12} \left\| u \right\|_{H^1_0(\OO)}, 
\end{equation}
that holds for all $v\in H^1_0(\OO)$.  In all the present work, we reserve the symbols $C$ and $K$ (eventually indexed) for constants depending on some or all the structural parameters of the problem. If the dependencies of such constants are relevant, they will be explicitly pointed out. The values of these constants may change within the same argument without relabeling. That being said, throughout this work, we shall make use of several different embedding theorems. As sometimes it shall be relevant to explicitly track down the value of some constants, for the sake of clarity, we precise some further notation. 
\begin{enumerate}[(i)] \itemsep0.3em
    \item The embedding $V \embed L^q(\OO)$ for all finite $q  \ge 1$ \andrea{in two dimensions and for $q \in [1,6]$ in three dimensions}, implies that
    \[
    \|u\|_{L^q(\OO)} \leq K_q \|u\|_{V}
    \]
    for all $u \in V$.
 \item The embedding $ L^\infty(\OO)\embed H$ implies that
    \[
    \|u\|_H \leq \andrea{|\OO|^\frac 12}\|u\|_{L^\infty(\OO)}
    \]
    for all $u \in L^\infty(\OO)$.
\end{enumerate}
    \subsection{Assumptions} The following assumptions are in order throughout the present work.
	\begin{enumerate}[label = \textbf{(A\arabic*)}] \itemsep0.5em
    \item \label{hyp:structural} The spatial domain $\OO$ is a bounded smooth domain in $\mathbb R^\andrea{d}$ \andrea{with $d \in \{2,3\}$}. The parameter $\varepsilon$ is strictly positive, and we set either $\alpha_d = 0$ and $\alpha_n = 1$ (\andrea{homogeneous} Neumann boundary conditions) or $\alpha_n = 0$ and $\alpha_d = 1$ (\andrea{homogeneous} Dirichlet boundary conditions).
	\item \label{hyp:psi} The function $\Psi:\mathbb R\to [0,+\infty]$ is convex, lower semicontinuous and
        $\operatorname{Int}(D(\Psi))=(0,L)$, where $L\in(0,+\infty]$.
        Moreover, the restriction $\Psi_{|(0,L)}\in C^2(0,L)$ and its subdifferential is single-valued, i.e., 
        \[
        \partial\Psi(s)=\{\Psi'(s)\} \qquad\forall\,s\in(0,L),
        \]
        \luca{
        and there exist constants $C_0,C_1>0$ such that 
        \[
        \Psi'(s)s\geq C_0s^2-C_1\quad\forall\,r\in(0,L).
        \]
        }
        Furthermore, the second derivative $\Psi''$ is convex and we assume that 
        \[
        \Psi''(0):=\lim_{s\to0^+}\Psi''(s)\in[0,+\infty],
        \]
        and, if $L < +\infty$,
        \[
        \Psi''(L):=\lim_{s\to L^-}\Psi''(s)\in[0,+\infty].
        \]
        \item \label{hyp:F} The function $F:\mathbb R\to\mathbb R$ is $\sqrt {C_F}$-Lipschitz-continuous for some $C_F>0$.
		\item \label{hyp:G} Setting $\Gamma:=\{v\in L^\infty(\OO):\: 0\leq v\leq L \text{ a.e.~in } \OO\}$ if $L < +\infty$ and $\Gamma:= \{v \in H: v \geq 0 \text{ a.e. in } \OO\}$ if $L=+\infty$,
        the function \luca{$G: \Gamma 
        \to \cL^2(U,H)$}
        satisfies: 
        \begin{enumerate}[(i)]
            \item  \label{hyp:Glip} there exists a constant $C_G > 0$ such that
            \[
            \|G(v_1)-G(v_2)\|_{\cL^2(U,H)}^2  \leq
           C_G\|v_1-v_2\|_H^2 \quad\forall\,v_1,v_2\in \Gamma;
            \]
            \item \label{hyp:Gbound} there exists a constant, which without loss of generality we set equal to $C_G$, such that
            \begin{alignat*}{2}
            \sum_{k=1}^{+\infty}\|\sqrt{\Psi''(v)}G(v)[e_k]\|_H^2
            & \leq \luca{
            C_G(1 + \|\Psi(v)\|_{L^1(\OO)})}
            \quad&&\forall\,v\in \Gamma,
            \end{alignat*}
            with the understanding that $\infty\cdot0:=0$.
        \end{enumerate}
        \item \label{hyp:mu} 
        Let $(Z, \mathscr Z)$ be a Blackwell space, 
        $\nu$ is a $\sigma$-finite positive measure
        on $(Z, \mathscr Z)$, and
        $\mu$ is a time-homogeneous Poisson random measure
        with intensity $\nu$.
		\item \label{hyp:J} The function 
        \luca{$J: \Gamma  \times Z \to H$} 
        satisfies:
        \begin{enumerate}[(i)]
            \item \label{hyp:Jlip} there exists a constant $C_J > 0$ such that
            \[
            \int_Z\|J(v_1,z)-J(v_2,z)\|_H^2\,\nu(\d z)  \leq
            C_J\|v_1-v_2\|_H^2 \quad\forall\,v_1,v_2\in  \Gamma;
            \]
            \item \label{hyp:delta} there exists a constant $\delta_J\in (0,1]$ such that
            \[
            \delta_J v \leq v+ J(v,z) \leq (1-\delta_J)L+\delta_J v
            \quad\text{a.e.~in } \OO, 
            \quad\forall\,v\in \Gamma,\quad\forall\,z\in Z
            \]
            if $L < +\infty$ and such that
            \[
            \delta_J v \leq v+ J(v,z) \leq C_J(1+v)
            \quad\text{a.e.~in } \OO, 
            \quad\forall\,v\in \Gamma,\quad\forall\,z\in Z
            \]
            if $L = +\infty$;
            \item \label{hyp:Jbound} there exists a constant, which without loss of generality we set equal to $C_J$, such that
            {\small
            \begin{alignat*}{2}
            \int_Z 
            \|\sqrt{\Psi''(\delta_Jv)}J(v,z)\|_H^2\:  \nu(\d z)
            & \leq C_J(1 + \|\Psi(v)\|_{L^1(\OO)})
            \quad&& \forall\,v\in \Gamma,\\
            \int_Z 
            \|\sqrt{\Psi''((1-\delta_J)L + \delta_Jv)}J(v,z)\|_H^2\:  \nu(\d z)
            & \leq C_J(1 + \|\Psi(v)\|_{L^1(\OO)})
            \quad&& \forall\,v\in \Gamma, \text{ if }L<+\infty, \\
            \int_Z 
            \|\sqrt{\Psi''(C_J(1+v))}J(v,z)\|_H^2\:  \nu(\d z)
            & \leq C_J(1 + \|\Psi(v)\|_{L^1(\OO)})
            \quad&& \forall\,v\in \Gamma, \text{ if }L=+\infty,
            \end{alignat*}}with the understanding that $\infty\cdot0:=0$.
        \end{enumerate}
	\end{enumerate}
    
\begin{remark}
    Let us point out some observations about the previous assumptions. Assumptions \ref{hyp:psi} and \ref{hyp:F} are classical hypotheses in the investigation of diffuse interface models. They encompass a wide class of potential by introducing a convex-concave splitting. The convex part, in particular, may exhibit singularities at the endpoints of $(0,L)$, allowing for thermodynamical potentials like the Flory--Huggins one (see \cite{Flory42, Huggins41}). The regular perturbation $F$ is most significant when $F$ is concave (otherwise it could be absorbed in $\Psi$). The competition of convex and concave potentials results in typical double-well profiles that are responsible for bistable dynamics. Assumption \ref{hyp:G} includes properties that are not new in the context of singular diffuse interface models (see, for instance, \cite{scarpa21, DPGS, DPGS2} for further insights). Finally, in the same spirit, Assumption \ref{hyp:J} ensures the correct compatibility conditions to grant existence of solutions. In particular, Assumption \ref{hyp:J}-\ref{hyp:Jbound} ensures that, after a jump, the variable does not exit the physically relevant range $(0,L)$ (in the case $L =+\infty$, it also prescribes a growth condition, in order not to have unnaturally ``large'' jumps).
\end{remark}
\begin{remark}
\label{rem:linear_growth}
    As a consequence of the Lipschitz continuity properties given in Assumptions \luca{\ref{hyp:F},
    \ref{hyp:G}-\ref{hyp:Glip} and \ref{hyp:J}-\ref{hyp:Jlip}, we infer the existence of positive constants 
    $\widetilde C_F$,
    $\widetilde{C}_G$ and $\widetilde{C}_J$ such that 
    \[
      |F(s)|^2\leq \widetilde C_F(1+|s|^2) \quad\forall\,s\in\mathbb R,
    \]
    }
    \[
        \|G(v)\|_{\cL^2(U,H)}^2  \leq
           \widetilde{C}_G\left(1 +\|v\|_H^2\right) \quad\forall\,v\in  \Gamma
            \]
    and 
    \[
    \int_Z\|J(v,z)\|_H^2\,\nu(\d z)  \leq
            \widetilde{C}_J\left( 1 +\|v\|_H^2\right) \quad\forall\,v\in \Gamma.
    \]
\end{remark}

\begin{remark}
\label{rem:Phi}
\andrea{It will turn out to be useful to introduce the function}
    \[
\Phi: \mathbb R \to \luca{(-\infty,+\infty]}, \qquad \Phi(r) = \Psi(r) + \int_0^r F(s) \: \d s, \qquad r \in \mathbb R.
\]
By possibly renominating the positive constants $C_0$ and $C_1$, it is not restrictive to assume that 
\[ 
\Phi'(r)r \ge C_0r^2-C_1 \quad\forall\,r\in(0,L).
\]
Since $\Phi$ is bounded from below, we may replace $\Phi$ by $\Phi + C$ for a suitable constant $C>0$ and assume, without loss of generality, that $\Phi \ge 0$.
\end{remark}
    \subsection{Main results} \andrea{
    In this subsection, we state the main results of this work. The first is a well-posedness result, also precising a notion of solution for problem \eqref{eq:strongform}.
    \begin{thm} \label{thm:wellposed}
        Let Assumption \ref{hyp:structural}-\ref{hyp:J} hold. Fix $p \geq 2$ and let $r := \min\{4,p\}$. Let further $u_0 \in L^p(\Omega, \cF_0; H)$ be such that 
        \[
        \Phi(u_0) \in L^\frac r2(\Omega, \cF_0; L^1(\OO)).
        \]
        Then, there exists a unique progressively measurable $H$-valued càdlàg process $u$ such that
        \begin{enumerate}[(i)]
            \item $u \in L^r(\Omega;L^\infty(0,T;H) \cap L^2(0,T;V))$;
            \item $u(\omega, t,x) \in (0,L)$ $\P \otimes \d t
            \otimes \d x$-almost surely;
            \item $\Psi(u) \in L^\frac r2(\Omega;L^\infty(0,T;L^1(\OO))$;
            \item $\Psi'(u) \in L^r(\Omega; L^2(0,T;H))$;
            \item $u(0) = u_0$;
            \item for every $v \in V$, it holds
            \begin{multline*}
                (u(t),\, v)_H + \eps\int_0^t\int_\OO \nabla u(s) \cdot \nabla v \: \d x \: \d s + \int_0^t\int_\OO (\Psi'(u(s))+F(u(s)))\,v \: \d x \: \d s \\
                = (u_0, v)_H + \left( \int_0^t G(u(s)) \: \d W(s),\, v\right)_H + \left( \int_{(0,t]} \int_Z J(u(s^-), z) \: \overline \mu(\d s, \d z),\, v\right)_H
            \end{multline*}
            for all $t \geq 0$, $\P$-almost surely.
        \end{enumerate}
        Moreover, given two initial conditions $u_{01}$ and $u_{02}$ satisfying the same assumptions listed above, letting $u_1$ and $u_2$ denote the corresponding solutions, the following continuous dependence estimate
        holds:
        \[
        \|u_1-u_2\|_{L^r(\Omega;L^\infty(0,T;H) \cap L^2(0,T;V))} \leq C\|u_{01}-u_{02}\|_{L^r(\Omega;H)}.
        \]
    \end{thm} \noindent
    The second result, instead, establishes existence, uniqueness and main properties of the invariant measure of the system.}
    \begin{thm} \label{thm:invariant}
Let Assumptions \ref{hyp:structural}-\ref{hyp:J} hold. \andrea{If $L=+\infty$, let further \[\min(\varepsilon,C_0)>K_2\left(\frac{\widetilde{C}_G}{2}+ \widetilde{C}_J \right).\]}Then, there exists at least 
one ergodic invariant measure supported on \andrea{the set}
\[
\mathcal{A}_{\mathrm{str}} := \left\{v \in V: \Psi(v) \in L^1(\OO), \,\Psi'(v) \in H  \right\}.
\]
Moreover, if Assumption \ref{hyp:structural} holds \andrea{with} $\alpha_n = 0$ and $\alpha_d = 1$ (Dirichlet boundary conditions) and 
\[
\varepsilon \rho_1> \frac{C_G}{2} + C_J+ \sqrt{C_F} \qquad  \text{if} \ L<+\infty
\]
or
\[
\min\left(\varepsilon \rho_1,\varepsilon,C_0\right)>\max\left(K_2\left(\frac{\widetilde{C}_G}{2}+ \widetilde{C}_J \right), \frac{C_G}{2}+C_J+\sqrt{C_F}\right) \qquad \text{if} \ L=+ \infty,
\]
then, the invariant measure is unique, \andrea{ergodic}, strongly and exponential\andrea{ly} mixing.
    \end{thm}
    \begin{remark}
        \andrea{It is expected that the actual support of the invariant measures of the system is even a subset of $\mathcal A_{\mathrm{str}}$. This can be shown by using higher-order estimates (as in \cite{Bertacco21}, for instance) that, however, need further integrability assumptions to hold on the parameters of the problem.}
        
    \end{remark}
	\section{Proof of Theorem~\ref{thm:wellposed}}
    \label{sec:wellposed}
    \andrea{This section is devoted to proving Theorem \ref{thm:wellposed}, i.e., a well-posedness result for system \eqref{eq:strongform}. Without loss of generality, we shall assume Neumann boundary conditions, namely $\alpha_d = 0$ and $\alpha_n = 1$. The proof in the Dirichlet case is analogous.}
    \subsection{An approximated problem.} \label{ssec:approximation}
    By virtue of Assumption \ref{hyp:psi}, we can apply the theory of maximal monotone operators (see for instance \cite{barbu-monot, brezis-monot}) to introduce a suitable regularization of the single-valued subdifferential $\partial\Psi$, which we shall denote by $\Psi'$, enabling in turn a regularization of $\Psi$. Indeed, since $\Psi$ is proper, convex and lower semicontinuous, the subdifferential $\Psi'$ can be identified with a maximal monotone graph in $\mathbb R \times \mathbb R$. It is therefore possible to consider the so-called Yosida approximation of $\Psi$, namely a one-parameter family of globally-defined and everywhere non-negative functions $\{\Psi_\lambda\}_{\lambda > 0}$ enjoying the following properties (cf. \cite{barbu-monot}):
\begin{enumerate}[label=(\alph*)]
	\item \label{pty:monot}$\Psi_\lambda$ is convex and converges to $\Psi$ pointwise in $\erre$ and monotonically increasing as $\lambda \to 0^+$;
	\item $\Psi_\lambda \in C^{1,1}(\mathbb{R})$ and the Lipschitz constant of $\Psi_\lambda'$ is $\frac{1}{\lambda}$;
	\item \label{pty:monotonicityder} $|\Psi_\lambda'|$ converges to $|\Psi'|$ pointwise in $\erre$ and monotonically increasing as $\lambda \to 0^+$,
	\item $\Psi_\lambda(0) = \Psi_\lambda'(0) = 0$ for all $\lambda > 0$.
\end{enumerate}
For any $\lambda > 0$, let $R_\lambda$ denote the non-expansive continuous resolvent operator
\[
R_\lambda: \mathbb{R} \to \mathbb{R}, \qquad R_\lambda(s) := (I+\lambda\Psi')^{-1}(s),
\qquad s\in\erre.
\]
Then, the Yosida approximation $\Psi_\lambda$ of $\Psi$ is precisely defined by
\[
\Psi_\lambda(s) = \Psi(R_\lambda(s)) + \dfrac{1}{2\lambda}|s - R_\lambda(s)|^2
\]
for all $s \in \mathbb R$. In order to setup an approximated version of \eqref{eq:strongform}, it is necessary to account for the fact that the operators $G$ and $J$ may not be well defined when evaluated on regular solutions, i.e., that approximated solutions may not belong to $\Gamma$.  For all $\lambda > 0$, define then the regularized operators
\[
\luca{G_\lambda: H \to \cL^2(U, H),} \qquad G_\lambda = G \circ R_\lambda
\]
as well as
\[
\luca{J_\lambda: H \times Z \to H,} \qquad J_\lambda(u, z) = J(R_\lambda(u), z) - \andrea{\lambda\Psi'(R_\lambda(u))} \qquad \forall \: (u,z) \in H \times Z.
\]
The action of the resolvent operator enables $G_\lambda$ and $J_\lambda(\cdot, z)$ to be globally defined operators over $H$ (and not only on $\Gamma$), for all $z \in Z$.
Before moving on, we prove an elementary, yet important preliminary result.
\begin{lem} \label{lem:properties}
    \andrea{The families of regularized functions $\{G_\lambda\}_{\lambda > 0}$ and $\{J_\lambda\}_{\lambda > 0}$ satisfy the following properties:
    \begin{enumerate}[(i)] \itemsep0.3em
        \item there exist two positive constants, that without loss of generality we still denote by $C_G$ and $C_J$, such that the two Lipschitz conditions
        \begin{align*}
            \|G_\lambda(v_1)-G_\lambda(v_2)\|^2_{\cL^2(U,H)} & \leq C_G\|v_1-v_2\|^2_H \\
            \int_Z\|J_\lambda(v_1,z)-J_\lambda(v_2,z) \|^2_{H}\: \nu(\d z) & \leq C_J\|v_1-v_2\|^2_H
        \end{align*}
        hold for all $\lambda > 0$ and all $v_1$, $v_2 \in V$;
        \item if $L <+\infty$, then for the same constant $\delta_J$ appearing in \ref{hyp:J}-\ref{hyp:delta} it holds
        \[
            \delta_JR_\lambda(v) \leq v + J_\lambda(v,z) \leq(1-\delta_J)L+\delta_JR_\lambda(v) \quad \text{a.e. in }\OO, \quad \forall \: v \in H, \quad \forall \: z \in Z,
        \]
        for all $\lambda > 0$. Analogously, if $L =+\infty$, we have
        \[
            \delta_JR_\lambda(v) \leq v + J_\lambda(v,z) \leq C_J(1+R_\lambda(v)) \quad \text{a.e. in }\OO, \quad \forall \: v \in H, \quad \forall \: z \in Z,
        \]
        for all $\lambda > 0$;
        \item there exist two positive constants, that without loss of generality we still denote by $C_G$ and $C_J$, such that the integral bounds
        \begin{align*}
            \sum_{k=1}^{+\infty}\|\sqrt{\Psi''_\lambda(v)}G_\lambda(v)[e_k]\|^2_H & \leq C_G(1+\|\Psi_\lambda(v)\|_{L^1(\OO)}), \\
            \int_Z \left\| \sqrt{\Psi''_\lambda(\delta_J R_\lambda(v))}J_\lambda(v,z) \right\|^2_H \: \nu(\d z) & \leq  C_J(1+\|\Psi_\lambda(v)\|_{L^1(\OO)}), \\
            \int_Z \left\| \sqrt{\Psi''_\lambda((1-\delta_J)L +\delta_J R_\lambda(v))}J_\lambda(v,z) \right\|^2_H \: \nu(\d z) & \leq  C_J(1+\|\Psi_\lambda(v)\|_{L^1(\OO)}) \quad \text{if $L<+\infty$}, \\
            \int_Z \left\| \sqrt{\Psi''_\lambda(C_J(1+R_\lambda(v))}J_\lambda(v,z) \right\|^2_H \: \nu(\d z) & \leq  C_J(1+\|\Psi_\lambda(v)\|_{L^1(\OO)}) \quad \text{if $L=+\infty$},
        \end{align*}
        hold for all $\lambda >0$ and $v \in H$.
    \end{enumerate}}
\end{lem}
\begin{proof}
Let $\lambda > 0$ be fixed and let us start from the Lipschitz conditions given \andrea{in item (i). Using Assumption \ref{hyp:G}-\ref{hyp:Glip}}, for any $v_1$ and $v_2$ in $H$, we have
\[
\begin{split}
    \|G_\lambda(v_1)-G_\lambda(v_2)\|^2_{\cL^2(U,H)} & = \|G(R_\lambda(v_1))-G(R_\lambda(v_2))\|^2_{\cL^2(U,H)} \\
    & \leq C_G\|R_\lambda(v_1)-R_\lambda(v_2)\|^2_H \\
    & \leq C_G\|v_1-v_2\|^2_H,
\end{split}
\]
since the resolvent operator is non-expansive in $H$. \andrea{An analogous argument shows the claim for $J_\lambda$, indeed Assumption \ref{hyp:J}-\ref{hyp:Jlip} gives
\[
\begin{split}
    \int_Z \|J_\lambda(v_1,z)-J_\lambda(v_2,z)\|_H^2 \: \nu(\d z) & \leq 2\int_Z\|J(R_\lambda(v_1), z)-J(R_\lambda(v_2),z)\|^2_H \: \nu(\d z) + 2\lambda^2\|\Psi'_\lambda(v_1)-\Psi'_\lambda(v_2)\|_H^2 \\
    & \leq 2(C_J+1)\|v_1-v_2\|^2_H
\end{split}
\]
as $\Psi'_\lambda$ is $\frac 1\lambda$-Lipschitz-continuous and the resolvent operator is non-expansive. The property in item (ii) can be shown as follows.
Without loss of generality, we show the case $L<+\infty$, as the other is analogous. Let $v \in V$ and $z \in Z$ be arbitrary. Then,
\[
\begin{split}
    v + J(R_\lambda(v),z) & = v - R_\lambda(v)+R_\lambda(v)+J(R_\lambda(v),z) \\
    & = \lambda\Psi_\lambda'(v) + R_\lambda(v)+J(R_\lambda(v),z),
\end{split}
\]
owing to the definition of the resolvent operator $R_\lambda$.
Observe that, by Assumption \ref{hyp:J}-\ref{hyp:delta}, we have
\[
\delta_J R_\lambda(v)\leq R_\lambda(v)+J(R_\lambda(v),z) \leq (1-\delta_J)L+\delta_J R_\lambda(v)
\]
and therefore
\[
\delta_J R_\lambda(v)\leq v + J(R_\lambda(v),z)-\lambda\Psi'_\lambda(v) \leq (1-\delta_J)L+\delta_J R_\lambda(v).
\]
The claim then holds recalling the definition of $J_\lambda$. 
Finally, we show the bounds listed in item (iii). Let $v \in H$.} First, using Assumption \ref{hyp:G}-\ref{hyp:Gbound}, we have  
\[
\begin{split}
    \sum_{k=1}^{+\infty}\left\|\sqrt{\Psi''_\lambda(v)}G_\lambda(v)[e_k]\right\|_H^2 & = \sum_{k=1}^{+\infty}\left\|\sqrt{\Psi''(R_\lambda(v))R'_\lambda(v)}G(R_\lambda(v))[e_k]\right\|_H^2 \\
    & \leq \sum_{k=1}^{+\infty}\|\sqrt{\Psi''(R_\lambda(v))}G(R_\lambda(v))[e_k]\|_H^2 \\
    & \leq C_G(1 + \|\Psi(R_\lambda(v))\|_{L^1(\OO)}) \\
    & \leq C_G(1 + \|\Psi_\lambda(v)\|_{L^1(\OO)}),
\end{split}
\]
since, given the definition of $\Psi_\lambda$, we have $\Psi_\lambda(s) \geq \Psi(R_\lambda(s))$ for all $s \in \mathbb R$. \andrea{For the remaining estimates, once again, we show one of the claimed inequalities, as the three arguments are very similar. Indeed, observing that $\delta_JR_\lambda(v) \in \operatorname{dom} \Psi''$ almost everywhere in $\OO$, 
\[
\begin{split}
    & \int_Z \left\|\sqrt{\Psi''_\lambda(\delta_J R_\lambda(v))}J_\lambda(v,z)\right\|^2_H \: \nu(\d z) 
    \\ 
    & \hspace{2cm}\leq 2\int_Z \|\sqrt{\Psi''(\delta_JR_\lambda(v))}J(R_\lambda(v),z)\|^2_H \: \nu(\d z) + 2\lambda^2 \left\|\sqrt{\Psi''_\lambda(\delta_JR_\lambda(v))}\Psi'_\lambda(v)\right\|^2_H \\
    & \hspace{2cm} \leq 2C_J(1+\|\Psi(R_\lambda(v))\|_{L^1(\OO)}) + 2\lambda\|\Psi'_\lambda(v)\|^2_H \\
    & \hspace{2cm} \leq 2C_J(1+\|\Psi_\lambda(v)\|_{L^1(\OO)}) + 2\lambda\|\Psi'_\lambda(v)\|^2_H.
\end{split}
\]
Recalling the definition of the Yosida approximation, we then have
\[
\Psi_\lambda(v) = \Psi(R_\lambda(v)) + \dfrac{|v-R_\lambda(v)|^2}{2\lambda} = \Psi(R_\lambda(v)) + \dfrac{\lambda}{2}|\Psi'_\lambda(v)|^2.
\]
Integrating this equality over $\OO$, recalling that $\Psi \circ R_\lambda \leq \Psi_\lambda$ and applying the triangular inequality yields then
\[
\lambda\|\Psi'_\lambda(v)\|^2_H \leq 4\|\Psi_\lambda(v)\|_{L^1(\OO)}
\]
and the claim follows.
}
\end{proof} \noindent
We are now in a position to consider the approximated problem 
\begin{equation} \label{eq:yosidastrong} \small
		\begin{cases}
			\d u_\lambda - \andrea{\eps} \Delta u_\lambda  \: \d t
            +\Psi'_\lambda(u_\lambda)\:\d t +
            F(u_\lambda)\:\d t = G_\lambda(u_\lambda) \: \d W + \displaystyle \int_Z J_\lambda(u_\lambda^-,z) \: \overline{\mu}(\d t,\, \d z) & \quad \text{in }  (0,T) \times\OO, \\
			\alpha_d u_\lambda + \alpha_n \partial_{\bn} u_\lambda = 0 & \quad \text{on } (0,T) \times \partial \OO, \\
			u_\lambda(\cdot, 0) = u_0 & \quad \text{in } \OO,
		\end{cases}
\end{equation}
and hence, we show that the above problem is well-posed. To this end, we set up a fixed point argument. For any fixed $\lambda > 0$, let 
\[
\mathcal X := L^r(\Omega; L^\infty(0,T;H) \cap L^2(0,T;V))
\]
where $r := \min \{4,p\}$ and $p$ is given in Theorem~\ref{thm:wellposed},
and consider for each $v \in \mathcal X$ the stochastic differential problem
\begin{equation} \label{eq:fpstrong} \small
		\begin{cases}
			\d w_\lambda - \andrea{\eps}\Delta w_\lambda  \: \d t
             = 
            -\Psi'_\lambda(v)\:\d t - F(v)\:\d t+ G_\lambda(v) \: \d W + \displaystyle \int_Z J_\lambda(v^-,z) \: \overline{\mu}(\d t,\, \d z) & \quad \text{in }  (0,T) \times\OO, \\
			\alpha_d w_\lambda + \alpha_n \partial_{\bn} w_\lambda = 0 & \quad \text{on } (0,T) \times \partial \OO, \\
			w(\cdot, 0) = u_0 & \quad \text{in } \OO,
		\end{cases}
\end{equation}
which, as a linear equation driven by \andrea{additive} noise, is well-posed, for instance, by checking the assumptions listed in \cite{scar-mar}. In particular, the solution map
\[
\andrea{\mathcal{S}} : \mathcal X \to \mathcal X, \qquad v \mapsto w_\lambda
\]
is well defined. Let now $v_1$ and $v_2$ be two elements of $\mathcal X$, and set
\[
w_{i, \lambda} = \andrea{\mathcal{S}} (v_i) \qquad \forall \: i \in \{1,2\}.
\]
Then, the stochastic process $w := w_{1, \lambda} - w_{2, \lambda}$ formally satisfies the stochastic partial differential system
\begin{equation} \label{eq:fpstrong2} \small
		\begin{cases}
			\d w -  \andrea{\eps}\Delta w  \: \d t
             = 
            -\left[\Psi'_\lambda(v_1)-\Psi'_\lambda(v_2)\right]\:\d t - \left[ F(v_1) - F(v_2)\right]\:\d t \\
            \qquad +  \left[ G_\lambda(v_1) - G_\lambda(v_2) \right] \: \d W + \displaystyle \int_Z \left[ J_\lambda(v^-_1,z) - J_\lambda(v^-_2,z) \right] \: \overline{\mu}(\d t,\, \d z) & \quad \text{in }  (0,T) \times\OO, \\
			\alpha_d w + \alpha_n \partial_{\bn} w = 0 & \quad \text{on } (0,T) \times \partial \OO, \\
			w(\cdot, 0) = 0 & \quad \text{in } \OO.
		\end{cases}
\end{equation}
For some $t \in (0,T)$, let us apply the It\^{o} formula to the squared $H$-norm of $w(t)$, i.e., precisely, to the functional
\[
w \mapsto \dfrac{1}{2}\|w\|^2_H 
\]
defined over $H$. Due to the presence of a semimartingale noise, the ordinary It\^{o} formula for twice differentiable functionals (as for instance in \cite[Theorem 4.32]{dapratozab}) is not complete: for the jump part we indeed resort to \cite[Theorem B.1]{itojump} and get 
\begin{multline} \label{eq:fp1}
     \dfrac{1}{2}\|w(t)\|^2_H +  \andrea{\eps}\int_0^t \|\nabla w(s)\|^2_{\b H} \: \d s \\  = - \int_0^t \left(w(s), \Psi'_\lambda(v_1(s))-\Psi'_\lambda(v_2(s))\right)_H \: \d s - \int_0^t \left(w(s), F(v_1(s))-F(v_2(s))\right)_H \: \d s \\
     + \int_0^t \left(w(s), G_\lambda(v_1(s))-G_\lambda(v_2(s)) \right)_H \: \d W(s) + \dfrac{1}{2} \int_0^t \|G_\lambda(v_1(s))-G_\lambda(v_2(s))\|^2_{\mathscr L^2(U, H)} \: \d s \\
     + \int_0^t \int_Z (w(s^-), J_\lambda(v_1(s^-),z) - J_\lambda(v_2(s^-),z))_H\: \overline{\mu}(\d s,\, \d z) \\
     + \int_0^t \int_Z \big[ \|w(s^-) + J_\lambda(v_1(s^-),z) - J_\lambda(v_2(s^-),z)\|^2_H \\ -\|w(s^-)\|^2_H - 2(w(s^-), J_\lambda(v_1(s^-),z) - J_\lambda(v_2(s^-),z))_H \big] \: \nu(\d z) \: \d s.
\end{multline}
Let us handle the various terms appearing to the right hand side of \eqref{eq:fp1}. First of all, by the Lipschitz properties listed in Assumptions \ref{hyp:psi} and \ref{hyp:F}, we have by the Cauchy--Schwarz and Young inequalities
\begin{equation} \label{eq:fp11}
    \begin{split}
        \left| \int_0^t \left(w(s), \Psi'_\lambda(v_1(s))-\Psi'_\lambda(v_2(s))\right)_H \: \d s + \int_0^t \left(w(s), F(v_1(s))-F(v_2(s))\right)_H \: \d s \right| \\
        \qquad \leq  \andrea{\sigma} \supp \|w(\tau)\|^2_H  + C_\sigma\left(\dfrac{1}{\lambda^2} + C_F\right)\int_0^t\|v_1(s)-v_2(s)\|^2_H \: \d s
    \end{split}
\end{equation} 
for any $ \andrea{\sigma} > 0$ and for some positive constant $C_\sigma$ only depending on $\sigma$. As for the It\^{o} trace term, we employ Lemma \ref{lem:properties} to get
\begin{equation} \label{eq:fp12}
    \dfrac{1}{2} \int_0^t \|G_\lambda(v_1(s))-G_\lambda(v_2(s))\|^2_{\mathscr L^2(U, H)} \: \d s  \leq \dfrac{C_G}{2} \int_0^t \|v_1(s)-v_2(s)\|^2_H \: \d s.
\end{equation}
Finally, the second-order term in the jump term, i.e., the last term in \eqref{eq:fp1}, after some elementary manipulations, equals 
\begin{equation} \label{eq:fp13}
    \begin{split}
        & \int_0^t \int_Z \big[ \|w(s^-) + J_\lambda(v_1(s^-),z) - J_\lambda(v_2(s^-),z)\|^2_H \\ & \hspace{2cm} -\|w(s^-)\|^2_H -\marghe{2} (w(s^-), J_\lambda(v_1(s^-),z) - J_\lambda(v_2(s^-),z))_H \big] \: \nu(\d z) \: \d s \\
        & = \int_0^t \int_Z \|J_\lambda(v_1(s^-),z) - J_\lambda(v_2(s^-),z)\|^2_H \: \nu(\d z) \: \d s \\
        & \leq C_J\int_0^t\|v_1(s)-v_2(s)\|^2_H \: \d s.
    \end{split}
\end{equation}
Taking into account \eqref{eq:fp11}-\eqref{eq:fp13} into \eqref{eq:fp1} (after selecting any $ \andrea{\sigma}$ > 0 small enough), \andrea{letting $r = \min \{4, p\}$ and then} taking $\frac {\andrea{r}}{2}$-th powers,  supremums in time and $\P$-expectations, we have,
\begin{multline} \label{eq:fp2}
     \E \supp \|w(\tau)\|^r_H + \E \left| \int_0^t \|\nabla w(s)\|^2_{\b H} \: \d s \right|^\frac r2 \\  \leq  C\left[ \E \left| \int_0^t\|v_1(s)-v_2(s)\|^2_H \: \d s \right|^\frac r2
     + \E \supp \left| \int_0^\tau \left(w(s), G_\lambda(v_1(s))-G_\lambda(v_2(s)) \right)_H \: \d W(s) \right|^\frac r2 \right.
     \\ \left. + \E \supp \left| \int_0^\tau \int_Z (w(s^-), J_\lambda(v_1(s^-),z) - J_\lambda(v_2(s^-),z))_H\: \overline{\mu}(\d s,\, \d z) \right|^\frac r2 \right]
\end{multline}
for a constant $C > 0$ depending on the structural parameters of the problem. We are only left with the stochastic integrals to the right hand side of \eqref{eq:fp2}. Indeed, they can be handled by means of a suitable Burkholder--Davis--Gundy inequality. In the first case, the version of the inequality for continuous real-valued martingales (see, for instance, \cite[Theorem 3.49]{pes-zab}) gives
\begin{equation} \label{eq:fp21}
\begin{split}
    & \E \supp \left| \int_0^\tau \left(w(s), G_\lambda(v_1(s))-G_\lambda(v_2(s)) \right)_H \: \d W(s) \right|^\frac r2 \\ & \hspace{2cm}  \leq C\E  \left| \int_0^t \|w(s)\|^2_H \|G_\lambda(v_1(s))-G_\lambda(v_2(s))\|^2_{\cL^2(U, H)} \: \d s \right|^\frac r4 \\ 
    & \hspace{2cm} \leq C\E  \left[ \supp \|w(\tau)\|^\frac{r}{2}_H \left| \int_0^t  \|G_\lambda(v_1(s))-G_\lambda(v_2(s))\|^2_{\cL^2(U, H)} \: \d s \right|^\frac r4 \right]\\
    & \hspace{2cm} \leq C\E \left[ \supp \|w(\tau)\|^\frac{r}{2}_H \left| \int_0^t  \|v_1(s)-v_2(s)\|^2_{H} \: \d s \right|^\frac r4 \right] \\
    & \hspace{2cm} \leq \dfrac{1}{4}\E \supp \|w(\tau)\|^r_H + C\E\left| \int_0^t \|v_1(s)-v_2(s)\|^2_{H} \: \d s \right|^\frac r2,
\end{split}
\end{equation}
while the similar result for the more general case of càdlàg processes \cite[Theorem 1]{mar-rock} gives, for the second integral, \andrea{as $r = \min\{4,p\} \leq 4$}, we have 
 
\begin{equation} \label{eq:fp22}
    \begin{split}
    & \E \supp \left| \int_0^\tau \int_Z (w(s^-), J_\lambda(v_1(s^-),z) - J_\lambda(v_2(s^-),z))_H\: \overline{\mu}(\d t,\, \d z) \right|^\frac r2 \\ & \hspace{2cm}  \leq C\E  \left| \int_0^t \|w(s^-)\|^2_H \int_Z \|J_\lambda(v_1(s^-),z) - J_\lambda(v_2(s^-),z)\|^2_{H} \:  \nu(\d z) \: \d s \right|^\frac r4 \\ 
    & \hspace{2cm} \leq C\E  \left[ \supp \|w(\tau^-)\|^\frac{r}{2}_H \left| \int_0^t  \|v_1(s^-)-v_2(s^-)\|^2_H \: \d s \right|^\frac r4 \right]\\
    & \hspace{2cm} \leq \dfrac{1}{4}\E \supp \|w(\tau)\|^r_H + C\E\left| \int_0^t \|v_1(s^-)-v_2(s^-)\|^2_{H} \: \d s \right|^{\frac r2},
\end{split}
\end{equation}
where we have also used the fact that
\[
\supp \|w(\tau^-)\|^r_H = \supp \lim_{y \to \tau^-} \|w(y)\|^r_H \leq \supp \|w(\tau)\|^r_H,
\]
tacitly extending $w$ to $w(0) = 0$ in a small left neighborhood of 0. Finally, owing to \eqref{eq:fp21} and \eqref{eq:fp22}, \eqref{eq:fp2} now reads
\begin{multline*}
         \E \supp \|w(\tau)\|^r_H + \E \left| \int_0^t \|\nabla w(s)\|^2_{\b H} \: \d s \right|^\frac r2 \\  \leq  C\left[ \E \left| \int_0^t\|v_1(s)-v_2(s)\|^2_H \: \d s \right|^\frac r2 +  \E \left| \int_0^t\|v_1(s^-)-v_2(s^-)\|^2_H \: \d s \right|^\frac r2\right]
\end{multline*}
which \andrea{eventually} leads to 
\begin{equation*}
         \E \supp \|w(\tau)\|^r_H + \E \left| \int_0^t \|\nabla w(s)\|^2_{\b H} \: \d s \right|^\frac r2 \\  \leq  C t^\frac r2  \E \supp \|v_1(\tau)-v_2(\tau)\|^r_H,
\end{equation*}
where $C > 0$ only depends on the structural parameters of the problem and, in particular, is independent of the time $t \geq 0$. Recalling the definition of $w$, we have indeed
\[
\|\andrea{\mathcal{S}}(v_1) - \andrea{\mathcal{S}}(v_2)\|^r_{\mathcal X} \leq Ct^\frac r2 \|v_1-v_2\|_{\mathcal X}^r.
\]
As the Lipschitz constant depends continuously on $t$ and vanishes in the limit $t \to 0^+$ for all fixed $p \geq 1$, there exists some final time $T_* = T_*(p)$ (\andrea{recall that $r = r(p)$}) depending also on the structural parameters of the problem such that
\[
 C^\frac 1rT_*^\frac 12  <1,
\]
and, therefore, the classical Banach fixed point theorem for contractions applies, yielding existence and uniqueness of a fixed point for the map $\andrea{\mathcal{S}}$, i.e., existence and uniqueness of a local-in-time probabilistically-strong solution to problem \eqref{eq:yosidastrong}. In order to retrieve a global solution, exploiting uniqueness of solutions, we can apply a standard patching argument (up to $\P$-indistinguishability), since the maximal time $T_*$ only depends on universal constants.
\subsection{Uniform estimates}
\label{sec:unif_est}
Having established the existence of approximated solutions, we now present an ensemble of uniform estimates with respect to the Yosida regularization parameter $\lambda > 0$.
\paragraph{\textit{First estimate.}} First of all, we want to apply the It\^{o} formula to the $H$-norm of $u_\lambda$. Following the procedure illustrated in the previous \andrea{sub}section, invoking again both \cite[Theorem 4.32]{dapratozab}) and \cite[Theorem B.1]{itojump} we have
\begin{multline} \label{eq:uniform11}
     \dfrac{1}{2}\|u_\lambda(t)\|^2_H + \andrea{\eps}\int_0^t \|\nabla u_\lambda(s)\|^2_{\b H} \: \d s = \frac 12\|u_\lambda(0)\|^2_H- \int_0^t \left(u_\lambda(s), \Psi'_\lambda(u_\lambda(s))\right)_H \: \d s - \int_0^t \left(u_\lambda(s), F(u_\lambda(s))\right)_H \: \d s \\
     + \int_0^t \left(u_\lambda(s), G_\lambda(u_\lambda(s))\: \d W(s) \right)_H  + \dfrac{1}{2} \int_0^t \|G_\lambda(u_\lambda(s))\|^2_{\mathscr L^2(U, H)} \: \d s \\
     + \int_0^t \int_Z (u_\lambda(s^-), J_\lambda(u_\lambda(s^-),z))_H\: \overline{\mu}(\d s,\, \d z) 
     \\+ \int_0^t \int_Z \big[ \|u_\lambda(s^-) + J_\lambda(u_\lambda(s^-),z)\|^2_H  -\|u_\lambda(s^-)\|^2_H - 2(u_\lambda(s), J_\lambda(u_\lambda(s^-),z))_H \big] \: \nu(\d z) \: \d s.
\end{multline}
In particular, recalling the monotonicity of $\Psi'_\lambda$ due to the convexity of $\Psi_\lambda$, we clearly have
\begin{equation} \label{eq:uniform12}
    \luca{\int_0^t \left(u_\lambda(s), \Psi'_\lambda(u_\lambda(s))\right)_H \: \d s \geq -C_1},    
\end{equation}
and by Assumption \ref{hyp:F} and Lemma \ref{lem:properties} we have
\begin{equation} \label{eq:uniform13}
   \left| \int_0^t \left(u_\lambda(s), F(u_\lambda(s))\right)_H \: \d s \right|\leq 
   \luca{\sqrt{\widetilde C_F} 
   \int_0^t \left(1+\|u_\lambda(s)\|_H\right)^2 \: \d s }
\end{equation}
as well as
\begin{equation} \label{eq:uniform14}
    \dfrac{1}{2} \int_0^t \|G_\lambda(u_\lambda(s))\|^2_{\mathscr L^2(U, H)} \: \d s \leq C\left( t + \int_0^t \|u_\lambda(s)\|^2_H \: \d s \right).
\end{equation}
For the last term in \eqref{eq:uniform11}, we obtain
\begin{equation} \label{eq:uniform15}
\begin{split}
     & \int_0^t \int_Z \big[ \|u_\lambda(s^-) + J_\lambda(u_\lambda(s^-),z)\|^2_H  -\|u_\lambda(s^-)\|^2_H - 2(u_\lambda(s^-), J_\lambda(u_\lambda(s^-),z))_H \big] \: \nu(\d z) \: \d s \\
     & \hspace{4cm} = \int_0^t \int_Z \|J_\lambda(u_\lambda(s^-), z)\|^2_H \: \nu(\d z) \: \d s \\
     & \hspace{4cm} \leq C\left( t + \int_0^t\|u_\lambda(s)\|^2_H \: \d s \right). 
\end{split}
\end{equation}
Collecting \eqref{eq:uniform12}-\eqref{eq:uniform15} in \eqref{eq:uniform11}, then taking $\frac r2$-th powers, \andrea{with $r = \min\{4,p\}$}, supremums in time and $\P$-expectations, we have,
\begin{multline} \label{eq:uniform16}
     \E \supp \|u_\lambda(\tau)\|^r_H + \E \left| \int_0^t \|\nabla u_\lambda(s)\|^2_{\b H} \: \d s \right|^\frac r2 \leq C\left( t^\frac r2 + \E \left| \int_0^t \|u_\lambda(s)\|^2_H \: \d s \right|^\frac{\marghe{r}}{2} \right) \\ 
     + \E \supp \left| \int_0^\tau  \left(u_\lambda(s), G_\lambda(u_\lambda(s)) \right)_H \: \d W(s) \right|^\frac r2 
     \\ + \E \supp \left| \int_0^t \int_Z (u_\lambda(s^-), J_\lambda(u_\lambda(s^-),z))_H\: \overline{\mu}(\d \andrea{s},\, \d z) \right|^\frac r2.
\end{multline}
In order to handle the stochastic terms in \eqref{eq:uniform16}, we mimic the computations in \eqref{eq:fp21} and \eqref{eq:fp22}, namely
\begin{equation} \label{eq:uniform17}
\begin{split}
    & \E \supp \left| \int_0^\tau  \left(u_\lambda(s), G_\lambda(u_\lambda(s)) \right)_H \: \d W(s) \right|^\frac r2 \\ & \hspace{1cm}  \leq C\E  \left| \int_0^t \|u_\lambda(s)\|^2_H \|G_\lambda(u_\lambda)\|^2_{\cL^2(U, H)} \: \d s \right|^\frac r4 \\ 
    & \hspace{1cm} \leq C\E  \left[ \supp \|
    u_\lambda(\tau)\|^\frac{r}{2}_H \left| \int_0^t  \|G_\lambda(u_\lambda(s))\|^2_{\cL^2(U, H)} \: \d s \right|^\frac r4 \right]\\
    & \hspace{1cm} \leq C\E \left[ \supp \|u_\lambda(\tau)\|^\frac{r}{2}_H \left( t^\frac r4 + \left|\int_0^t  \|u_\lambda(s)\|^2_{H} \: \d s \right|^\frac r4 \right)\right] \\
    & \hspace{1cm} \leq \dfrac{1}{4}\E \supp \|u_\lambda(\tau)\|^r_H + C\left( t^\frac r2 + \E\left|\int_0^t  \|u_\lambda(s)\|^2_{H} \: \d s \right|^\frac r2 \right),
\end{split}
\end{equation}
and, by the same token, \andrea{using also the fact that $r \leq 4$},
\begin{equation} \label{eq:uniform18}
    \begin{split}
    & \E \supp \left| \int_0^t \int_Z (u_\lambda(s^-), J_\lambda(u_\lambda(s^-),z))_H\: \overline{\mu}(\d \andrea{s},\, \d z) \right|^\frac r2\\ 
    & \hspace{1cm}  \leq C\E  \left| \int_0^t \|u_\lambda(s^-)\|^ 2_H \int_Z \|J_\lambda(u_\lambda(s^-),z)\|^2_{H} \:  \nu(\d z) \: \d s \right|^\frac r4 \\ 
    & \hspace{1cm} \leq C\E  \left[ \supp \|u_\lambda(\tau^-)\|^\frac{r}{2}_H \left( t^\frac r4 + \left|\int_0^t  \|u_\lambda(s)\|^2_{H} \: \d s \right|^\frac r4 \right) \right]\\
    & \hspace{1cm} \leq \dfrac{1}{4}\E \supp \|u_\lambda(\tau)\|^r_H + C\left( t^\frac r2 +\E \left|\int_0^t  \|u_\lambda(s)\|^2_{H} \: \d s \right|^\frac r2 \right),
\end{split}
\end{equation}
Combining \eqref{eq:uniform16}, \eqref{eq:uniform17} and \eqref{eq:uniform18} and exploiting the Gronwall lemma, we obtain that there exists a constant $C_1 > 0$ independent of $\lambda > 0$ such that
\begin{equation} \label{eq:uniform1final}
\|u_\lambda\|_{L^\andrea{r}(\Omega; L^\infty(0,T;H)) \cap L^\andrea{r}(\Omega; L^2(0,T;V))} \leq C_1
\end{equation}
for all $\lambda >0$ \andrea{and $r = \min\{4,p\}$.}
\paragraph{\textit{Second estimate.}}
A second useful estimate is given by applying the It\^{o} lemma to 
\[
\mathcal H: H \to \mathbb R \qquad \mathcal H(v) := \int_\OO \Phi_\lambda(v) \: \d x
\]
evaluated at $v = u_\lambda(s)$ for any $s \geq 0$, where, in turn, we set
\[
\Phi_\lambda: \mathbb R \to \mathbb R \qquad \Phi_\lambda(r) = \Psi_\lambda(r) + \int_0^r F(s) \: \d s
\]
for all $r \in \mathbb R$. \andrea{Without loss of generality, we assume that $\Phi_\lambda \geq 0$. Indeed, it is well known that for $\lambda \in (0,\lambda_0)$ it holds
\[
\Psi_\lambda(s) \geq \dfrac{1}{4\lambda_0}s^2-C_{\lambda_0},
\]
and therefore, as Assumption \ref{hyp:F} implies 
\luca{$|F(s)| \leq \sqrt{\widetilde C_F}(1+s)$} for all $s \in \erre$, we also have
\[
\luca{\Phi_\lambda(s) \geq \left(\dfrac{1}{4\lambda_0}-\sqrt{\widetilde C_F} \right)s^2-C_{\lambda_0}-\widetilde C_F}
\]
for all $s \in \mathbb R$ and all $\lambda \in (0,\lambda_0)$. Choosing any $\lambda_0>0$ sufficiently small yields a bound from below for $\Phi_\lambda$. The application of the It\^{o} lemma }yields (compare with \cite[Subsection 4.2]{Bertacco21}), \andrea{thanks to \cite[Theorem.~B.1]{itojump}}
{\footnotesize
\begin{multline} \label{eq:uniform21}
    \int_\OO \Phi_\lambda(u_\lambda(t))\: \d x + \andrea{\eps}\int_0^t (\Phi_\lambda''(u_\lambda(s)), |\nabla u_\lambda(s)|^2)_H \: \d s + \int_0^t \|\Phi'_\lambda(u_\lambda(s))\|^2_H \: \d s \\
    = \int_\OO \Phi_\lambda(u_\lambda(0))\: \d x + \dfrac{1}{2}\int_0^t \sum_{k=0}^{+\infty} (\Phi''_\lambda(u_\lambda(s)), |G_\lambda(u_\lambda(s))[e_k]|^2)_H \: \d s \\ + \int_0^t (\Phi'_\lambda(u_\lambda(s)), G_\lambda(u_\lambda(s)))_H \: \d W(s)
    + \int_0^t \int_Z (\Phi_\lambda'(u_\lambda(s^-)), J_\lambda(u_\lambda(s^-),z))_H\: \overline{\mu}(\d t,\, \d z) 
     \\+ \int_0^t \int_Z \left[ \int_\OO \Phi_\lambda(u_\lambda(s^-) + J_\lambda(u_\lambda(s^-),z)) \: \d x - \int_\OO \Phi_\lambda(u_\lambda(s^-)) \: \d x  - \int_\OO \Phi_\lambda'(u_\lambda(s^-)) J_\lambda(u_\lambda(s^-),z)\: \d x \right] \: \nu(\d z) \: \d s.
\end{multline}}On the left hand side, we observe that
\begin{equation} \label{eq:uniform22}
    \begin{split}
        \int_0^t (\Phi_\lambda''(u_\lambda(s)), |\nabla u_\lambda(s)|^2)_H \: \d s & = \int_0^t (\Psi_\lambda''(u_\lambda(s)), |\nabla u_\lambda(s)|^2)_H \: \d s + \int_0^t (F'(u_\lambda(s)), |\nabla u_\lambda(s)|^2)_H \: \d s \\
        & \geq  - C_F \int_0^t \|\nabla u_\lambda(s)\|_\b H^2 \: \d s
    \end{split}
\end{equation}
by convexity of $\Psi_\lambda$ and the fact that $F'$ is well-defined almost everywhere and bounded, as a  consequence of Assumption \ref{hyp:F} and Rademacher's theorem. On the right hand side, \andrea{arguing similarly and owing to Lemma \ref{lem:properties}} we have
\andrea{
\begin{equation} \label{eq:uniform23}
\begin{split}
    & \dfrac{1}{2}\int_0^t \sum_{k=1}^{+\infty} (\Phi''_\lambda(u_\lambda(s)), |G_\lambda(u_\lambda(s))[e_k]|^2)_H \: \d s \\
    & \hspace{2cm} \leq C\left( \int_0^t \sum_{k=1}^{+\infty} \|\sqrt{\Psi''_\lambda(u_\lambda(s))}|G_\lambda(u_\lambda(s))[e_k]|\|_H^2 \: \d s + \int_0^t\|G_\lambda(u_\lambda(s))\|^2_{\cL^2(U,H)} \: \d s\right)\\
    & \hspace{2cm} \leq C\left(t + \int_0^t\|\Psi_\lambda(u_\lambda(s))\|_{L^1(\OO)}\: \d s + \int_0^t\|u_\lambda(s)\|^2_{H} \: \d s\right).
\end{split}
\end{equation}
Moreover, owing once again to Lemma \ref{lem:properties}
{\footnotesize
\begin{equation} \label{eq:uniform22-2}
    \begin{split}
        & \left| \int_0^t \int_Z \left[ \int_\OO \Phi_\lambda(u_\lambda(s^-) + J_\lambda(u_\lambda(s^-),z)) \: \d x - \int_\OO \Phi_\lambda(u_\lambda(s^-)) \: \d x  - \int_\OO \Phi_\lambda'(u_\lambda(s^-)) J_\lambda(u_\lambda(s^-),z)\: \d x \right] \: \nu(\d z) \: \d s \right| \\
        & \hspace{1cm} \leq C\left| \int_0^t \int_Z \int_\OO \int_0^1 \Phi_\lambda''(u_\lambda(s^-) + \theta J_\lambda(u_\lambda(s^-),z))|J_\lambda(u_\lambda(s^-),z)|^2 \: \d \theta \: \d x\: \nu(\d z) \: \d s \right| \\
        & \hspace{1cm} \leq C\left| \int_0^t \int_Z \int_\OO \int_0^1 \Psi''_\lambda(u_\lambda(s^-) + \theta J_\lambda(u_\lambda(s^-),z))|J_\lambda(u_\lambda(s^-),z)|^2 \: \d\theta\: \d x\: \nu(\d z) \: \d s \right| \\
        & \hspace{4cm} +  C\left| \int_0^t \int_Z \int_\OO \int_0^1F'(u_\lambda(s^-) + \theta J_\lambda(u_\lambda(s^-),z))|J_\lambda(u_\lambda(s^-),z)|^2 \:\d\theta \: \d x\: \nu(\d z) \: \d s \right| \\
        & \hspace{1cm} \leq C\left| \int_0^t \int_Z \int_\OO \int_0^1 \Psi''_\lambda(u_\lambda(s^-) + \theta J_\lambda(u_\lambda(s^-),z))|J_\lambda(u_\lambda(s^-),z)|^2 \: \d \theta \: \d x\: \nu(\d z) \: \d s \right| \\
        & \hspace{4cm}+ C\left| \int_0^t \int_Z \int_\OO |J_\lambda(u_\lambda(s^-),z)|^2 \: \d x\: \nu(\d z) \: \d s \right|.
    \end{split}
\end{equation}}Without loss of generality, we shall} assume that $L$ is finite. \andrea{The proof in the case $L=+\infty$ is analogous. It is clear how to handle the second term in \eqref{eq:uniform22-2}. For the first one, we invoke Lemma \ref{lem:properties} once more. Indeed, observe that
\[
\begin{split}
    u_\lambda(s^-) +\theta J_\lambda(u_\lambda(s^-),z) = (1-\theta)u_\lambda(s^-)+\theta[u_\lambda(s^-)+J_\lambda(u_\lambda(s^-),z)]
\end{split}
\]
for any $\theta \in [0,1]$, and in turn, Lemma \ref{lem:properties} implies there exists a constant $\sigma \in [0,1]$ such that
\[
u_\lambda(s^-) + \theta J_\lambda(u_\lambda(s^-),z) = (1-\theta)u_\lambda(s^-)+ \sigma\theta \delta_JR_\lambda(u_\lambda(s^-)) + (1-\sigma)\theta\left[(1-\delta_J)L + \delta_J R_\lambda(u_\lambda(s^-))\right]
\]
and using the convexity of $\Psi''_\lambda$ we eventually arrive at 
\begin{multline*}
    \left| \int_0^t \int_Z \int_\OO \int_0^1 \Psi''_\lambda(u_\lambda(s^-) + \theta J_\lambda(u_\lambda(s^-),z))|J_\lambda(u_\lambda(s^-),z)|^2 \: \d \theta \: \d x\: \nu(\d z) \: \d s \right|  \\ \leq \left| \int_0^t \int_Z \int_\OO \Psi''_\lambda( u_\lambda(s^-))|J_\lambda(u_\lambda(s^-),z)|^2 \: \d x\: \nu(\d z) \: \d s \right| \\ +\left| \int_0^t \int_Z \int_\OO \Psi''_\lambda(\delta_J R_\lambda(u_\lambda(s^-)))|J_\lambda(u_\lambda(s^-),z)|^2 \: \d x\: \nu(\d z) \: \d s \right| \\ + \left| \int_0^t \int_Z \int_\OO \Psi''_\lambda((1-\delta_J)L + \delta_J R_\lambda(u_\lambda(s^-)))|J_\lambda(u_\lambda(s^-),z)|^2 \: \d x\: \nu(\d z) \: \d s \right|
\end{multline*}
The last two terms can be controlled directly by means of Lemma \ref{lem:properties}. As for the first one, recall that
\[
\delta_JR_\lambda(v) \leq R_\lambda(v) = (1-\delta_J)R_\lambda(v)+\delta_JR_\lambda(v) \leq (1-\delta_J)L+\delta_JR_\lambda(v)
\]
as the range of $R_\lambda$ lies in the effective domain of $\Psi'$, namely $(0,L)$. Since, by definition of Yosida approximation, $\Psi''_\lambda \leq \Psi''\circ R_\lambda$ and $\Psi''$ is convex, we deduce that
\[ \small
\begin{split}
    & \left| \int_0^t \int_Z \int_\OO \Psi''_\lambda( u_\lambda(s^-))|J_\lambda(u_\lambda(s^-),z)|^2 \: \d x\: \nu(\d z) \: \d s \right| \\
    & \hspace{0.3cm}\leq C\left[ \left| \int_0^t \int_Z \int_\OO \Psi''_\lambda( u_\lambda(s^-))|J(R_\lambda(u_\lambda(s^-)),z)|^2 \: \d x\: \nu(\d z) \: \d s \right| + \lambda^2 \left| \int_0^t  \int_\OO \Psi''_\lambda( u_\lambda(s^-))|\Psi'_\lambda(u_\lambda(s^-))|^2 \: \d x \: \d s \right|\right] \\
    & \hspace{0.3cm}\leq C\left[\left| \int_0^t \int_Z \int_\OO \Psi''( R_\lambda(u_\lambda(s^-)))|J(R_\lambda(u_\lambda(s^-)),z)|^2 \: \d x\: \nu(\d z) \: \d s \right|+ \lambda \left| \int_0^t  \int_\OO |\Psi'_\lambda(u_\lambda(s^-))|^2 \: \d x \: \d s \right| \right]\\
    & \hspace{0.3cm} \leq C\left[ \left| \int_0^t \int_Z \int_\OO \Psi''(\delta_J R_\lambda(u_\lambda(s^-)))|J(R_\lambda(u_\lambda(s^-)),z)|^2 \: \d x\: \nu(\d z) \: \d s \right| \right. \\
    & \hspace{0.6cm} \left. + \left| \int_0^t \int_Z \int_\OO \Psi''((1-\delta_J)L + \delta_J R_\lambda(u_\lambda(s^-)))|J(R_\lambda(u_\lambda(s^-)),z)|^2 \: \d x\: \nu(\d z) \: \d s \right| + \lambda \left| \int_0^t  \int_\OO |\Psi'_\lambda(u_\lambda(s^-))|^2 \: \d x \: \d s \right| \right]
\end{split}
\]
On account of all of the above, we conclude by Lemma \ref{lem:properties} and Assumption \ref{hyp:J} that
\begin{multline*}
    \left| \int_0^t \int_Z \int_\OO \int_0^1 \Psi''_\lambda(u_\lambda(s^-) + \theta J_\lambda(u_\lambda(s^-),z))|J_\lambda(u_\lambda(s^-),z)|^2 \: \d\theta\: \d x\: \nu(\d z) \: \d s \right| \\  \leq C\left[1+\int_0^t\|\Psi_\lambda(u_\lambda(s^-))\|_{L^1(\OO)} \: \d s + \lambda\int_0^t\|\Psi'_\lambda(u_\lambda(s^-))\|^2_H \: \d s\right].
\end{multline*}
}Collecting the above estimates in \eqref{eq:uniform21}, raising the result to the power $\frac r2$, \andrea{with $r = \min \{p,4\}$}, taking supremums in time and $\P$-expectations, we get
{\footnotesize
\begin{multline} \label{eq:uniform24}
    \E \supp \left| \int_\OO \Phi_\lambda(u_\lambda(t))\: \d x \right|^\frac r2 + \andrea{(1-C\lambda)^\frac r2 \E \left| \int_0^t \|\Phi'_\lambda(u_\lambda(s))\|^2_H \: \d s \right|^\frac r2} \\
    \leq 6^\frac r2 C^\frac r2 \left[ 1 + \E \left| \int_\OO \Phi_\lambda(u_\lambda(0))\: \d x\right|^\frac r2 + \E \left| \int_\OO \| \nabla u_\lambda(s)\|^2_{\b H} \: \d s\right|^\frac r2 + \E \left| \int_0^t \|\Phi_\lambda(u_\lambda(s))\|_{L^1(\OO)}  \: \d s\right|^\frac r2 \right. \\ \left. + \E \supp \left| \int_0^\tau (\Phi'_\lambda(u_\lambda(s)), G_\lambda(u_\lambda(s)))_H \: \d W(s) \right|^\frac r2
    + \E \supp \left| \int_0^\tau \int_Z (\Phi_\lambda'(u_\lambda(s^-)), J_\lambda(u_\lambda(s^-),z))_H\: \overline{\mu}(\d t,\, \d z) \right|^\frac r2\right],
\end{multline}}where we reconstructed $\Phi_\lambda$ from $\Psi_\lambda$ by algebraic manipulations and using the \textit{First estimate} (recall that any primitive of $F$ is quadratically bounded). \andrea{In \eqref{eq:uniform24}, we kept $C$ independent of both $\lambda$ and $r$ for better clarity. The value of $C > 0$ is then the same on both the left and the right hand side.} Finally, we need to control the stochastic terms, following the same strategy employed in the \textit{First estimate,} namely
\begin{equation} \label{eq:stoch1}
\begin{split}
    & \E \supp \left| \int_0^\tau  \left(\Phi'_\lambda(u_\lambda(s)), G_\lambda(u_\lambda(s)) \right)_H \: \d W(s) \right|^\frac r2 \\ & \hspace{1cm}  \leq C\E  \left| \int_0^t \|\Phi'_\lambda(u_\lambda(s))\|^2_H \|G_\lambda(u_\lambda(s))\|^2_{\cL^2(U, H)} \: \d s \right|^\frac r4 \\ 
    & \hspace{1cm} \leq C\E  \left[ \supp \|G_\lambda(u_\lambda(\tau))\|^\frac r2_{\cL^2(U, H)} \left| \int_0^t  \|\Phi'_\lambda(u_\lambda(s))\|^2_H  \: \d s \right|^\frac r4 \right]\\
    & \hspace{1cm} \leq C\E \left[ \supp \left( 1 + \|u_\lambda(\tau)\|_H^\frac r2\right)\left| \int_0^t  \|\Phi'_\lambda(u_\lambda(s))\|^2_H  \: \d s \right|^\frac r4 \right] \\
    & \hspace{1cm} \leq \andrea{\sigma} \E \left| \int_0^t  \|\Phi'_\lambda(u_\lambda(s))\|^2_H  \: \d s \right|^\frac r2 + C\left( 1 +  \E \supp \|u_\lambda(\tau)\|^r_H \right),
\end{split}
\end{equation}
where $\andrea{\sigma} > 0$ is arbitrarily small (but fixed). Analogously, \andrea{owing still to the fact that $r \leq 4$}, we have
\begin{equation} \label{eq:stoch2}
    \begin{split}
    & \E \supp \left| \int_0^\tau \int_Z (\Phi_\lambda'(u_\lambda(s^-)), J_\lambda(u_\lambda(s^-),z))_H\: \overline{\mu}(\d s,\, \d z) \right|^\frac r2\\ 
    & \hspace{1cm}  \leq C\E  \left| \int_0^t \|\Phi_\lambda'(u_\lambda(s^-))\|^2_H \int_Z \|J_\lambda(u_\lambda(s^-),z)\|^2_{H} \:  \nu(\d z) \: \d s \right|^\frac r4 \\ 
    & \hspace{1cm} \leq C\E  \left[ \supp \left| \int_Z \|J_\lambda(u_\lambda(\tau^-),z)\|^2_{H} \:  \nu(\d z)\right|^\frac r4\left|\int_0^t  \|\Phi'_\lambda(u_\lambda(s))\|^2_H  \: \d s\right|^\frac r4 \right]\\
    & \hspace{1cm} \leq \andrea{\sigma} \E \left| \int_0^t  \|\Phi'_\lambda(u_\lambda(s))\|^2_H  \: \d s \right|^\frac r2 + C\left( 1 +  \E \supp \|u_\lambda(\tau)\|^r_H \right).
\end{split}
\end{equation}
\andrea{Let us remark that, in all Burkholder--Davis--Gundy-type estimates for jump terms}, in order to \andrea{allow for} $r > 4$, \andrea{controls} in $L^p(\Omega; L^p(0,T;H))$ \andrea{are generally} needed. \andrea{In particular, this would require some strengthened form of Assumption \ref{hyp:J}}. In the above computations, we used the properties of $G_\lambda$ and $J_\lambda$ given by Lemma \ref{lem:properties}.
Consequently, \andrea{for sufficiently small values of $\sigma$ and $\lambda$}, \eqref{eq:stoch1} and \eqref{eq:stoch2} jointly with \eqref{eq:uniform24} imply that
\begin{multline*}
    \E \supp \left| \int_\OO \Phi_\lambda(u_\lambda(t))\: \d x \right|^\frac r2 + \E \left| \int_0^t \|\Phi'_\lambda(u_\lambda(s))\|^2_H \: \d s \right|^\frac r2 \\
    \leq C \left[ 1 + \E \left| \int_\OO \Phi_\lambda(u_\lambda(0))\: \d x\right|^\frac r2 +\E \left| \int_0^t \|\Phi_\lambda(u_\lambda(s))\|_{L^1(\OO)}  \: \d s\right|^\frac r2 \right],
\end{multline*}
owing also to the \textit{First estimate.} The Gronwall lemma then implies that there exists a constant $C_2 > 0$ independent of $\lambda$ such that
\[
\|\Phi'_\lambda(u_\lambda))\|_{L^\andrea{r}(\Omega;L^2(0,T;H))} \leq C_2.
\]
It is easy to check by comparison that 
\[
\|\Psi'_\lambda(u_\lambda))\|_{L^\andrea{r}(\Omega;L^2(0,T;H))} \leq C_2
\]
for a possibly different constant that we do not relabel.
\paragraph{\textit{Third estimate.}} The final estimate aims to show that the sequence $\{u_\lambda\}_{\lambda > 0}$ is a Cauchy sequence in \andrea{the space
\[
L^r(\Omega;L^\infty(0,T;H)) \cap L^r(\Omega;L^2(0,T;V))
\]
where $r = \min \{4,p\}$}, along values of $\lambda$ that converge to 0. Therefore, it must be convergent to some limit. This estimate follows closely the computations done in the \textit{First estimate} (see also Subsection \ref{ssec:uniq}) and exploits the resolvent identity to get a dependence on $\lambda$. Let $\lambda_1 > 0$ and let $\lambda_2 > 0$ such that $\lambda_1 > \lambda_2$. Applying the It\^{o} lemma to the $H$-norm of $\utwo -\uone$ we get
\begin{multline} \label{eq:uniform31}
     \dfrac{1}{2}\|\uone(t)-\utwo(t)\|^2_H + \int_0^t \|\nabla \uone(s)-\nabla \utwo(s)\|^2_{\b H} \: \d s \\  = - \int_0^t \left(\uone(s)-\utwo(s), \Psi_{\lambda_1}'(\uone(s))-\Psi'_{\lambda_2}(\utwo(s))\right)_H \: \d s \\- \int_0^t \left(\uone(s)-\utwo(s),\, F(\uone(s))-F(\utwo(s))\right)_H \: \d s \\
     + \int_0^t \left(\uone(s)-\utwo(s), (G_{\lambda_1}(\uone(s))-G_{\lambda_2}(\utwo(s))) \: \d W(s) \right)_H  \\ + \dfrac{1}{2} \int_0^t \|G_{\lambda_1}(\uone(s))-G_{\lambda_2}(\utwo(s))\|^2_{\mathscr L^2(U, H)} \: \d s \\
     + \int_0^t \int_Z (\uone(s^-)-\utwo(s^-), J_{\lambda_1}(\uone(s^-),z) - J_{\lambda_2}(\utwo(s^-),z))_H\: \overline{\mu}(\d s,\, \d z) \\
     + \int_0^t \int_Z \big[ \|\uone(s^-)-\utwo(s^-) + J_{\lambda_1}(\uone(s^-),z) - J_{\lambda_2}(\utwo(s^-),z)\|^2_H \\ -\|\uone(s^-)-\utwo(s^-)\|^2_H - 2(\uone(s)-\utwo(s), J_{\lambda_1}(\uone(s^-),z) - J_{\lambda_2}(\utwo(s^-),z))_H \big] \: \nu(\d z) \: \d s.
\end{multline}
Recalling the resolvent identity
\[
\lambda\Psi_\lambda'(s) + R_\lambda(s) = s
\]
for all $s \in \mathbb R$ and $\lambda > 0$, we observe that
\[
\begin{split}
    &\left(\uone(s)-\utwo(s), \Psi'_{\marghe{\lambda_1}}(\uone(s))-\Psi'_{\marghe{\lambda_2}}(\utwo(s))\right)_H \\
    & \hspace{1cm}= \left(\lambda_1\Psi'_{\lambda_1}(\uone(s)) - \lambda_2 \Psi'_{\lambda_2}(\utwo(s)), \Psi_{\lambda_1}'(\uone(s))-\Psi'_{\lambda_2}(\utwo(s))\right)_H \\
    & \hspace{2cm} + \left(R_{\lambda_1}(\uone(s)) - R_{\lambda_2}(\utwo(s)), \Psi_{\lambda_1}'(\uone(s))-\Psi'_{\lambda_2}(\utwo(s))\right)_H \\
    & \hspace{1cm}\geq -(\lambda_1+\lambda_2)(\Psi_{\lambda_1}'(\uone(s)),\,\Psi'_{\lambda_2}(\utwo(s)))_H \\
    & \hspace{1cm} \geq -\dfrac{\lambda_1+\lambda_2}{2}\|\Psi_{\lambda_1}'(\uone(s))\|^2 -\dfrac{\lambda_1+\lambda_2}{2}\|\Psi_{\lambda_2}'(\utwo(s))\|^2.
\end{split}  
\]
The second term, thanks to the Lipschitz continuity of $F$, satisfies
\[
\left| \left(\uone(s)-\utwo(s),\, F(\uone(s))-F(\utwo(s))\right)_H  \right| \leq C\|\uone(s)-\utwo(s)\|^2_H,
\]
with $C > 0$ independent of $\lambda$. Thanks to the Lipschitz continuity of $G$ and recalling that $G_\lambda = G \circ R_\lambda$, we get
\[
\|G_{\lambda_1}(\uone(s))-G_{\lambda_2}(\utwo(s))\|^2_{\mathscr L^2(U, H)} \leq \|R_{\lambda_1}(\uone(s))-R_{\lambda_2}(\utwo(s))\|^2_H
\]
and the only deterministic term left is the last one. By the same token, the jump term satisfies
\[
\begin{split}
    & \int_Z \big[ \|\uone(s^-)-\utwo(s^-) + J_{\lambda_1}(\uone(s^-),z) - J_{\lambda_2}(\utwo(s^-),z)\|^2_H \\ &\hspace{1cm }-\|\uone(s^-)-\utwo(s^-)\|^2_H - 2(\uone(s)-\utwo(s), J_{\lambda_1}(\uone(s^-),z) - J_{\lambda_2}(\utwo(s^-),z))_H \big] \: \nu(\d z) \\
    & = \int_Z\|J_{\lambda_1}(\uone(s^-),z) - J_{\lambda_2}(\utwo(s^-),z)\|^2 \: \nu(\d z) \\
    & \leq C\|R_{\lambda_1}(\uone(s^-))-R_{\lambda_2}(\utwo(s^-))\|^2_H.
\end{split}
\]
Employing the above estimates in \eqref{eq:uniform31}, raising to the power $\frac r2$, taking supremums in time and $\P$-expectations, we arrive at
\begin{multline} \label{eq:uniform32}
     \E \supp \|\uone(t)-\utwo(t)\|^r_H + \E\left| \int_0^t \|\nabla \uone(s)-\nabla \utwo(s)\|^2_{\b H} \: \d s \right|^\frac r2 \\  \leq C\left[\E \left| \int_0^t \|\uone(s)-\utwo(s)\|^2_H \: \d s \right|^\frac r2 + \E \left| \int_0^t \|R_{\lambda_1}(\uone(s))-R_{\lambda_2}(\utwo)\|^2_H \: \d s\right|^\frac r2 \right. \\ + \dfrac{\lambda_1+\lambda_2}{2}\E \left|\int_0^t\|\Psi'_{\lambda_1}(\uone(s))\|_H^2 \: \d s \right|^\frac r2  + \dfrac{\lambda_1+\lambda_2}{2}\E \left|\int_0^t\|\Psi'_{\lambda_2}(\utwo(s))\|_H^2 \: \d s \right|^\frac r2 \\ 
     + \E \supp \left| \int_0^\tau \left(\uone(s)-\utwo(s), (G_{\lambda_1}(\uone(s))-G_{\lambda_2}(\utwo(s))) \: \d W(s) \right)_H \right|^\frac r2  \\ \left.
     + \E \supp \left| \int_0^\tau \int_Z (\uone(s^-)-\utwo(s^-), J_{\lambda_1}(\uone(s^-),z) - J_{\lambda_2}(\utwo(s^-),z))_H\: \overline{\mu}(\d s,\, \d z) \right|^\frac r2 \right]
\end{multline}
For the stochastic terms in \eqref{eq:uniform32}, we follow the same computations of the \textit{First estimate} so that, eventually,
\begin{multline} \label{eq:uniform33}
     \E \supp \|\uone(t)-\utwo(t)\|^r_H + \E\left| \int_0^t \|\nabla \uone(s)-\nabla \utwo(s)\|^2_{\b H} \: \d s \right|^\frac r2 \\  \leq C\left[\E \left| \int_0^t \|\uone(s)-\utwo(s)\|^2_H \: \d s \right|^\frac r2 + \E \left| \int_0^t \|R_{\lambda_1}(\uone(s))-R_{\lambda_2}(\utwo)\|^2_H \: \d s\right|^\frac r2 \right. \\ \left. + \dfrac{\lambda_1+\lambda_2}{2}\E \left|\int_0^t\|\Psi'_{\lambda_1}(\uone(s))\|_H^2 \: \d s \right|^\frac r2  + \dfrac{\lambda_1+\lambda_2}{2}\E \left|\int_0^t\|\Psi'_{\lambda_2}(\utwo(s))\|_H^2 \: \d s \right|^\frac r2 \right]
\end{multline}
The last two expectations in \eqref{eq:uniform33} are uniformly bounded in $\lambda_1$ and $\lambda_2$. For the resolvent difference, observe that by summing and subtracting $R_{\lambda_1}(\utwo(s))$ 
\[
\begin{split}
     \|R_{\lambda_1}(\uone(s))-R_{\lambda_2}(\utwo(s))\|^2_H & \leq 2\|\uone(s)-\utwo(s)\|^2_H + 2\|R_{\lambda_1}(\utwo(s))-R_{\lambda_2}(\utwo(s))\|^2_H \\
     & \leq C\left( \|\uone(s)-\utwo(s)\|^2_H + \|\marghe{\lambda_1}\Psi'_{\lambda_1}(\utwo(s))-\marghe{\lambda_2}\Psi_{\lambda_2}'(\utwo(s))\|_H^2\right) \\
     & \leq C\left( \|\uone(s)-\utwo(s)\|^2_H + \lambda_1^2\|\Psi'_{\lambda_2}(\utwo(s))\|_H^2+\lambda_2^2\|\Psi'_{\lambda_2}(\utwo(s))\|^2_H\right)
\end{split}
\]
where we used the fact that the Yosida approximation converges monotonically from below as $\lambda \to 0^+$ (recall that $\lambda_1 > \lambda_2$). Then, we arrive at
\begin{multline} \label{eq:uniform34}
     \E \supp \|\uone(t)-\utwo(t)\|^r_H + \E\left| \int_0^t \|\nabla \uone(s)-\nabla \utwo(s)\|^2_{\b H} \: \d s \right|^\frac r2 \\  \leq C\left[\E \left| \int_0^t \|\uone(s)-\utwo(s)\|^2_H \: \d s \right|^\frac r2 +\lambda_1 + \lambda_2 + \lambda_1^2+\lambda_2^2\right]
\end{multline}
The Gronwall lemma implies the result, taking the limits $\lambda_1 \to 0^+$ and $\lambda_2 \to 0^+$.
\subsection{Passage to the limit $\lambda \to 0^+$.} In this section, we infer convergence properties from the estimates shown in the previous Subsection. All convergences illustrated hereafter are always intended as $\lambda \to 0^+$ along an arbitrary decreasing sequence. First of all, since $\{u_\lambda\}_{\lambda > 0}$ is a Cauchy sequence in a Banach space, we conclude that
\[
u_\lambda \to u \quad \text{in }L^r(\Omega;L^\infty(0,T;H)) \cap L^r(\Omega;L^2(0,T;V))
\]
\andrea{where $r = \min \{4,p\}$.} By reflexivity and the uniform bound given in the \textit{Second estimate}, we have
\[
\Phi'_\lambda(u_\lambda) \rightharpoonup \xi \quad \text{in } L^r(\Omega;L^2(0,T;H)).
\]
It is our aim now to prove that $\xi = \Psi'(u) + F(u)$. Observe that since $F$ is a Lipschitz-continuous function we immediately achieve that
\[
\E \left| \int_0^t \|F(u_\lambda(s)) - F(u(s))\|^2_H \: \d s\right|^\frac r2 \leq C\E \left| \int_0^t \|u_\lambda(s)) - u(s)\|^2_H \: \d s\right|^\frac r2
\]
giving
\[
F(u_\lambda) \to F(u) \quad \text{in } L^r(\Omega;L^2(0,T;H)).
\]
In particular, then, by comparison it is immediate to deduce that
\[
\Psi'_\lambda(u_\lambda) \rightharpoonup \xi - F(u) \quad \text{in } L^r(\Omega;L^2(0,T;H)).
\]
Since $\Psi'_\lambda = \Psi' \circ R_\lambda$, we now study the convergence of the sequence $R_\lambda(u_\lambda)$. The triangle inequality and the resolvent identity give
\[
\begin{split}
    \|R_\lambda(u_\lambda(s)) - u(s)\|^2_H & \leq 2\|R_\lambda(u_\lambda(s))-u_\lambda(s)\|^2_H + 2\|u_\lambda(s)-u(s)\|^2_H \\
    & = 2\lambda^2\|\Psi'_\lambda(u_\lambda(s))\|^2_H + 2\|u_\lambda(s)-u(s)\|^2_H
\end{split}
\]
implying that
\[
R_\lambda(u_\lambda) \to u \quad \text{in } L^r(\Omega;L^2(0,T;H)),
\]
thanks to the \textit{Second estimate}.
Now, observe that
\[
\E\int_0^t(\Psi'(R_\lambda(u_\lambda(s)))-\Psi'(y(s)), R_\lambda(u_\lambda(s))-y(s))_H \: \d s\geq 0
\]
for all $y \in L^2(\Omega; L^2(0,T;H))$ such that $\Psi'(y) \in L^2(\Omega; L^2(0,T;H))$ and $\Psi'$ is maximal monotone. Passing to the limit yields
\[
(\xi - F(u)-\Psi'(y), u-y)_{L^2(\Omega;L^2(0,T;H)} \geq 0
\]
and maximality implies $\xi -F(u)= \Psi'(u)$, as claimed. Observe, that, in turn, this implies that
\[
u \in V \cap \Gamma \quad \text{a.e. in }\Omega \times \OO \times (0,T)
\]Finally, we are left to deal with the stochastic integrals. Let us start from the It\^{o} term. The key tool is, once again, the Burkholder-Davis-Gundy inequality, recalling that $G_\lambda = G \circ R_\lambda$. Indeed,
\[
\begin{split}
    & \E \supp \left\| \int_0^\tau G(R_\lambda(u_\lambda(s))) - G(u(s)) \: \d W(s) \right\|_H^r  \\
    & \hspace{2cm}\leq C\E\left|\int_0^t \| G(R_\lambda(u_\lambda(s))) - G(u(s))\|^2_{\cL^2(U,H)} \: \d s\right|^\frac r2 \\
    & \hspace{2cm} \leq C\E\left|\int_0^t \| R_\lambda(u_\lambda(s)) - u(s)\|^2_H \: \d s\right|^\frac r2 \\
    & \hspace{2cm}\leq C\left[ \lambda^2 \E \left| \int_0^t \|\Psi'_\lambda(u_\lambda(s))\|^2_H \: \d s \right|^\frac r2 + \E \left| \int_0^t \|u_\lambda(s)-u(s)\|^2_H \: \d s \right|^\frac r2\right]
\end{split}
\]
giving
\[
G_\lambda(u_\lambda)  \cdot W\to G(u) \cdot  W\quad \text{in } L^r(\Omega;L^\infty(0,T;H)).
\]
A similar argument works for the jump term, as, since $r \leq 4$, the Burkholder-Davis-Gundy inequality works similarly. Therefore,
\[
\int_Z J_\lambda(u_\lambda^-, z) \: \overline \mu(\d t, \d z) \to \int_Z J(u^-,z)\:\overline
\mu(\d t, \d z) \quad \text{in } L^r(\Omega;L^\infty(0,T;H)).
\]
The proven convergences are also enough to infer that the limit process $u$ is indeed a probabilistically-strong solution to \eqref{eq:strongform}. The proof of existence is complete.
\subsection{Uniqueness of solutions.}
\label{ssec:uniq}
Finally, we prove that the solution to problem \eqref{eq:strongform} is pathwise unique. This is, essentially, a third iteration of the computations shown in Subsection \ref{ssec:approximation}. In particular, let $u_{01}$ and $u_{02}$ be elements of $L^p(\Omega; \andrea{H})$ complying with the assumptions of Theorem \ref{thm:wellposed}. For $i \in \{1,2\}$, let $u_i$ denote the solution to \eqref{eq:strongform} having $u_{0i}$ as initial state. Setting
\[
u := u_1 - u_2, \qquad u_0 := u_{01} - u_{02}
\]
we have that formally
\begin{equation} \label{eq:uniqueness1} \small
		\begin{cases}
			\d u - \andrea{\eps}\Delta u  \: \d t
             = 
            -\left[\Psi'(u_1)-\Psi'(u_2)\right]\:\d t - \left[ F(u_1) - F(u_2)\right]\:\d t \\
            \qquad +  \left[ G(u_1) - G(u_2) \right] \: \d W + \displaystyle \int_Z \left[ J(u^-_1,z) - J(u^-_2,z) \right] \: \overline{\mu}(\d t,\, \d z) & \quad \text{in }  (0,T) \times\OO, \\
			\alpha_d u + \alpha_n \partial_{\bn} u = 0 & \quad \text{on } (0,T) \times \partial \OO, \\
			u(\cdot, 0) = u_0 & \quad \text{in } \OO.
		\end{cases}
\end{equation}
For simplicity, in the following we set without loss of generality $\alpha_d = 0$ and $\alpha_n = 1$. Arguing as in Subsection \ref{ssec:approximation}, we apply the It\^{o} lemma to the functional
\[
v \mapsto \dfrac 12 \|v\|^2_H
\]
evaluated at $v = u(t)$ for any $t > 0$. This yields
\begin{multline} \label{eq:unique1}
     \dfrac{1}{2}\|u(t)\|^2_H + \andrea{\eps}\int_0^t \|\nabla u(s)\|^2_{\b H} \: \d s \\  = \frac12\|u_{01}-u_{02}\|^2_H- \int_0^t \left(u(s), \Psi'(u_1(s))-\Psi'(u_2(s))\right)_H \: \d s - \int_0^t \left(u(s), F(u_1(s))-F(u_2(s))\right)_H \: \d s \\
     + \int_0^t \left(u(s), (G(u_1(s))-G(u_2(s))) \: \d W(s) \right)_H  + \dfrac{1}{2} \int_0^t \|G(u_1(s))-G(u_2(s))\|^2_{\mathscr L^2(U, H)} \: \d s \\
     + \int_0^t \int_Z (u(s^-), J(u_1(s^-),z) - J(u_2(s^-),z))_H\: \overline{\mu}(\d s,\, \d z) \\
     + \int_0^t \int_Z \big[ \|u(s^-) + J(u_1(s^-),z) - J(u_2(s^-),z)\|^2_H \\ -\|u(s^-)\|^2_H - 2(u(s), J(u_1(s^-),z) - J(u_2(s^-),z))_H \big] \: \nu(\d z) \: \d s.
\end{multline}
In the same fashion, we employ all the known assumptions as well as the convexity properties of $\Psi$ to deduce that
\begin{equation} \label{eq:unique2}
    \int_0^t \left(u(s), \Psi'(u_1(s))-\Psi'(u_2(s))\right)_H \: \d s \geq 0
\end{equation}
as well as 
\begin{equation} \label{eq:unique22}
    \left| \int_0^t \left(u(s), F(u_1(s))-F(u_2(s))\right)_H \: \d s \right| \leq C \int_0^t \|u(s)\|^2_H \: \d s
\end{equation}
and
\begin{equation} \label{eq:unique3}
     \dfrac{1}{2} \int_0^t \|G(u_1(s))-G(u_2(s))\|^2_{\mathscr L^2(U, H)} \: \d s \leq C\int_0^t \|u(s)\|^2_H \: \d s.
\end{equation}
Once again, for the last term of \eqref{eq:unique1} we have
\begin{equation} \label{eq:unique4}
    \begin{split}
        & \int_0^t \int_Z \big[ \|u(s^-) + J(u_1(s^-),z) - J(u_2(s^-),z)\|^2_H \\ & \hspace{2cm} -\|u(s^-)\|^2_H - (u(s^-), J(u_1(s^-),z) - J(u_2(s^-),z))_H \big] \: \nu(\d z) \: \d s \\
        & \hspace{1cm} = \int_0^t \int_Z \|J(u_1(s^-),z) - J(u_2(s^-),z)\|^2_H \: \nu(\d z) \: \d s \\
        &  \hspace{1cm}\leq C\int_0^t\|u(s)\|^2_H \: \d s.
    \end{split}
\end{equation}
Therefore, collecting \eqref{eq:unique2}-\eqref{eq:unique4} in \eqref{eq:unique1}, taking $\frac r2$-th powers, \andrea{with $r = \min\{4,p\}$}, supremums in time and $\P$-expectations, we get
\begin{multline} \label{eq:unique5}
     \E \supp\|u(\tau)\|^r_H + \E \left| \int_0^t \|\nabla u(s)\|^2_{\b H} \: \d s \right|^\frac r2 \leq C\E \left| \int_0^t \|u(s)\|^2\: \d s\right|^\frac r2 \\ 
     + \E \supp \left|\int_0^\tau \left(u(s), G(u_1(s))-G(u_2(s)) \right)_H \: \d W(s) \right|^\frac r2
     \\+ \E \supp \left| \int_0^t \int_Z (u(s^-), J(u_1(s^-),z) - J(u_2(s^-),z))_H\: \overline{\mu}(\d t,\, \d z)\right|^\frac r2.
\end{multline}
On account of the computations in \eqref{eq:fp21} and \eqref{eq:fp22}, we can conclude by the Gronwall lemma that
\[
\|u_1 - u_2\|_{L^r(\Omega; L^\infty(0,T;H) \cap L^2(0,T;V))} \leq C\|u_{01}-u_{02}\|_{L^r(\Omega;H)}
\]
and the claim follows setting $u_{01} = u_{02}$ $\P$-almost surely.
\section{Proof of Theorem \ref{thm:invariant}}
\label{sec:invariant}
\andrea{This section is devoted to showing Theorem \ref{thm:invariant}, i.e., to performing a long-time analysis of the stochastic equation \eqref{eq:strongform} in terms of existence and uniqueness of invariant measures and mixing properties. As we shall see, it will be useful to introduce some \textit{ad-hoc} notation that we illustrate hereafter.}
\subsection{Preliminaries}
\andrea{First, we consider} the set
\[
\mathcal{A} := \{v \in H\ : \Psi(v) \in L^1(\OO)\}=\bigcup_{n \in \mathbb{N}}\left\{ v \in H\ : \int_{\OO}\Psi(v) \le n \right\},
\]
which implies that $\mathcal{A}$ is a convex Borel subset of $H$ since $\Psi$ is convex and lower semicontinuous. 
\andrea{Let us endow the set $\mathcal{A}$ with the structure of metric space inherited by the space $H$. The resulting separable metric space will be denoted by $(\mathcal A, \mathsf d)$.} 
\andrea{The symbol} $\mathscr{B}(\mathcal{A})$ \andrea{denotes} the $\sigma$-algebra of all Borel subsets of $\mathcal{A}$; \andrea{moreover,} for a set $A \in \mathscr{B}(\mathcal{A})$, we denote by $A^c$ its complement. By $\mathscr{P}(\mathcal{A})$ we denote the set of all probability measures on $(\mathcal{A}, \mathscr{B}(\mathcal{A}))$. Moreover, $\mathcal{B}_b(\mathcal{A})$ and $\mathcal{C}_b(\mathcal{A})$ will stand for the space of all measurable bounded functions from $\mathcal{A}$ to $\mathbb{R}$ and the space of continuous bounded functions from $\mathcal{A}$ to $\mathbb{R}$, respectively. \andrea{Instead, by $\operatorname{Lip}(\A)$ we denote the space of Lipschitz continuous functions from $\A$ to $\mathbb R$, endowed with its classical norm $\|\cdot\|_{\mathrm{Lip}}$.} We \andrea{are now in a position} to rigorously introduce the family of transition operators  associated to equation \eqref{eq:strongform}. 
For every $x \in \mathcal{A}$, we denote by $u^x$ the unique solution to \eqref{eq:strongform} with initial datum $x$. \andrea{Accordingly,} for every $t\ge 0$, we set $u^x(t):=u(t\andrea{,\,}x)$ \andrea{to denote its} value at time $t$.
\andrea{Observe that for all $x \in \mathcal A$, Theorem \ref{thm:wellposed} establishes} that $u^x(t) \in \mathcal{A}$ for every $t>0$, $\mathbb{P}$-a\andrea{lmost surely}, and that $u^x(t)$ is $\mathcal{F}_t/ \mathscr{B}(\mathcal{A})$-measurable. Hence, given $t>0$, the law of $u^x(t)$ on $\mathcal{A}$ is defined as the pushforward measure 
\[
\cL_t: \mathcal{A} \times \mathscr{B}(\mathcal{A}) \rightarrow [0,1], \qquad \cL_t(x, A)=\operatorname{Law}_{\mathbb{P}}(u^x(t))\andrea{(A)}=\mathbb{P}(u^x(t) \in A), \qquad \forall \:x \in \mathcal{A}, \ A \in \mathscr{B}(\mathcal{A}).
\]
\luca{For every $A\in\cB(\mathcal A)$
one has that
the map $L^2(\Omega,\cF_t; H)\to [0,1]$ given
by $X\mapsto\operatorname{Law}_{\P}(X)(A)$, $X\in L^2(\Omega, \cF_t; H)$, is measurable: this can be proved 
by density arguments, by approximating $\mathbbm{1}_A$
through Lipschitz-continuous functions and by the dominated convergence theorem.
Hence, recalling that 
$x\mapsto u^x(t)$ is continuous from $\mathcal A$ to $X\in L^2(\Omega, \cF_t; H)$, we have that the map $x \mapsto \cL_t(x,A)$ is $\mathscr{B}(\mathcal{A})$-measurable for any $A \in \mathscr{B}(\mathcal{A})$.} This implies that the \andrea{map} $\cL_t$ is a transition kernel from $(\mathcal{A}, \mathscr{B}(\mathcal{A}))$ into itself \andrea{for every $t \ge 0$}. Hence, we can define the family of operators $P=(P_t)_{t \ge 0}$ as 
\[
P_t:\mathcal{B}_b(\mathcal{A}) \rightarrow \mathcal{B}_b(\mathcal{A}), \qquad (P_t\varphi)(x):=\int_{\mathcal{A}}\varphi(y)\, \cL_t(x,{\rm d}y)=\mathbb{E}[\varphi(u^x(t))], \qquad \forall \:x \in \mathcal{A},\ \varphi \in \mathcal{B}_b(\mathcal{A}).
\]
\andrea{The family of operators $P$ is well defined: indeed,} $P_t\varphi \in \mathcal{B}_b(\mathcal{A})$ for every $\varphi \in \mathcal{B}_b(\mathcal{A})$ since $\cL_t$ is a transition kernel.
\andrea{Let us point out once more} that, due to the nonlinear nature of the problem, the solution of equation \eqref{eq:strongform} \andrea{belongs to $\mathcal A$}. \andrea{Therefore,} the transition semigroup can only make sense as a family of operators acting on $\mathcal{B}_b(\mathcal{A})$ and not on $\mathcal{B}_b(H)$ as in the more classical cases. Furthermore, thanks to \cite[Theorem 9.30]{pes-zab} one has that the unique solution of \eqref{eq:strongform} is \andrea{also} a Markov process. 
Therefore, we deduce that the family of operators $\andrea{P}$ is a Markov semigroup, 
namely $P_{t+s}=P_tP_s$ for every $s$ and $t\ge0$. \andrea{The following definition precises finally the notion of invariant measure.}
\begin{defin} \label{def:invariant}
  An invariant measure for the transition semigroup $P$ is a probability measure $\gamma\in\cP(\mathcal{A})$ such that 
  \[
  \int_{\mathcal{A}} \varphi(x)\,\gamma(\d x) = \int_{\mathcal{A}} P_t\varphi(x)\,\gamma(\d x)
  \]
  for all $t\geq0$ and for all $\varphi\in \mathcal C_b(\mathcal{A})$.
\end{defin} \noindent
\andrea{Throughout this section, integration with respect to $x$ will typically denote integration over $\mathcal A$ with respect to some probability measure, instead of integration over the space domain $\OO$.}
\subsection{Existence of an invariant measure}
\andrea{First, we show that the transition semigroup $P$ admits indeed at least one invariant measure, in the sense of Definition \ref{def:invariant}}. The proof relies on an adaptation of the Krylov--Bogoliubov theorem to the case of complete separable metric spaces, see \cite[Theorem 3.3]{SZ}. \andrea{As customary when applying the} Krylov--Bogoliubov theorem, we need to \andrea{show} that the semigroup $P$ is Feller and the tightness property \andrea{for time-averaged laws}. \andrea{The Feller property} follows directly from 
the continuous dependence of the solution \andrea{from} the initial data. Indeed, let $t>0$ and $\varphi \in \mathcal{C}_b(\mathcal{A})$ be fixed.
Given a sequence $\{x_n\}_{n \in \enne} \subset \mathcal{A}$
which converges in $\mathcal{A}$ to
$x\in \mathcal{A}$ as $n \rightarrow +\infty$, \andrea{we need to prove that} the sequence $\{P_t\varphi(x_n)\}_{n \in \enne}$ converges to
$P_t\varphi(x) $ as $n \rightarrow +\infty$. \andrea{Owing to the continuous dependence estimate in Theorem \ref{thm:wellposed}}, we have that
\begin{equation*}
\|u^{x_n}(t)-u^x(t)\|_{L^2(\Omega; H)} \leq 
\|u^{x_n}-u^x\|_{L^2(\Omega; L^\infty(0,t;H))} \le C_t\|x_n-x\|_{H}.
\end{equation*}
It follows that, as $n\to+\infty$, for every $t \ge 0$, \andrea{the convergence} $u^{x_n}(t)\rightarrow u^x(t)$ \andrea{holds} in $L^2(\Omega;H)$, 
hence also in probability.
In turn, this implies that $\varphi(u^{x_n}(t))\rightarrow \varphi(u^x(t))$ in probability
by the continuity of $\varphi$. The boundedness of $\varphi$ and the Vitali theorem yield in particular that $\varphi(u^{x_n}(t))\rightarrow \varphi(u^x(t))$ \andrea{even} in $L^1(\Omega)$, and thus 
\begin{equation*}
|(P_t\varphi)(x_n)-(P_t \varphi)(x)| \le \mathbb{E}\left[\left |\varphi(u^{x_n}(t))- \varphi(u^x(t))\right| \right] \rightarrow 0,
\end{equation*}
as $n \rightarrow +\infty$. This shows that $P$ is Feller. \andrea{Next,} we prove that $P$ satisfies the tightness property \andrea{given, for instance, in} \cite[Theorem 3.3]{SZ}. To this end, we consider the solution to \eqref{eq:strongform} stemming from a fixed initial condition $x_0\in\mathcal A$.
Consider then the family of measures 
$(\gamma_t)_{t>0}\subset \cP(\mathcal{A})$ defined by 
\begin{equation*}
\andrea{\gamma_t:\mathscr B(\mathcal A) \to [0,1]}, \qquad A \mapsto \frac 1t \int_0^t (P_s\mathbbm{1}_{A})(x_0)\, {\rm d}s
=\frac 1t \int_0^t \mathbb{P}\left(u^{x_0}(t) \in A \right)\, {\rm d}s, \qquad \forall \: A \in \cB(\mathcal{A}), \ t>0.
\end{equation*}
\andrea{Let us prove that the family} is tight.
Let $B_n^V$ be the closed ball in $V$ of radius $n \in \mathbb{N}$ \andrea{and centered at the origin, and set $\widehat B_n^V:= B_n^V \cap \mathcal{A}$ to denote the part of the ball in $\mathcal A$.
Since the embedding $V \hookrightarrow H$ is compact, the set $\widehat B_n^V$ is a compact subset of $\mathcal{A}$. Exploiting} Lemma~\ref{lem:apriori1} and the Chebychev inequality we infer, for any $t\geq1$,
\begin{align*}
\gamma_t\left(\left(\widehat B_n^V\right)^c\right)
&=\frac 1t \int_0^t (P_s\mathbbm{1}_{\left(\widehat B_n^V\right)^c})(x_0)\, {\rm d}s=
\frac 1t \int_0^t \mathbb{P}\left(\|u^{x_0}(s)\|^2_V \ge n^2 \right)\, {\rm d}s
\\
&\le \frac{1}{tn^2}\int_0^t \mathbb{E}\left[\|u^{x_0}(s)\|^2_V\right]\, {\rm d}s
 \le \frac{C}{n^2},
\end{align*}
where $C$ is a positive constant depending on the parameters $\widetilde{C}_G,\, \widetilde{C}_J,\, C_1$ \andrea{and} $\OO$ in the case $L<+ \infty$ and on the parameters $\widetilde{C}_G,\, \widetilde{C}_J,\, C_1,\, C_0,\, \varepsilon$ \andrea{and} $K_2$ in the case $L=+ \infty$. \andrea{Choosing $n \in \enne$ sufficiently large yields the claim and completes the proof.}
\subsection{Support of invariant measures}
\label{ssec:inv_support}
\andrea{Having established the} existence of invariant measures, we focus 
here on some \andrea{of their} qualitative properties. \andrea{In particular, }we shall prove integrability properties \andrea{that} in turn provide information on their support, \andrea{which is contained in} a \andrea{strict subset of} $\A$. To this end we introduce the set 
\begin{equation}
  \label{A_sr}
  \A_{\text{str}}:=\left\{x\in \A \cap V: \ \Phi'(x)\in H\right\} \andrea{\subset \A}.
\end{equation}
\andrea{In the definition above, recall that $\Phi$ is defined as in Remark \ref{rem:Phi}}. Exploiting the lower semicontinuity of $|\Phi'|$, \andrea{it is possible to} show that $\A_{\text{str}}$ is a Borel subset of $H$, hence of $\A$. Let $\gamma \in \cP(\mathcal{A})$ be an invariant measure for the transition semigroup $P$. \andrea{For the sake of clarity, we divide the argument into four steps.}\\[0.3\baselineskip]
{\sc Step 1.}
\andrea{First, we show} that there exists a positive constant $C$ \andrea{independent of $\gamma$} such that 
\begin{equation}
\label{est_inv_meas_H}
\int_{\mathcal{A}} \|x\|^2_H \: \andrea{\gamma(\d x)} \le C
\end{equation}
\andrea{for any invariant measure $\gamma$. To this end,} we consider the mapping 
\[
\Theta:\mathcal{A} \rightarrow [0, + \infty), \qquad x \mapsto \|x\|^2_H, \qquad  \forall \: x \in \mathcal{A}
\]
and its approximations $\{\Theta_n\}_{n\in\enne}$, defined for every $n\in\enne$ as
\[
\Theta_n:\mathcal{A}\to[0,n^2], \qquad 
\Theta_n:x \mapsto 
\begin{cases}
\|x\|^2_H & \text{if} \ x \in B_n^H \cap \mathcal{A},
\\
n^2 & \text{otherwise},
\end{cases}
\quad x\in\mathcal{A},
\]
where we set $B_n^H$ as the closed ball of radius $n$ in $H$ \andrea{centered at zero}. \andrea{It is immediate to show} that $\Theta_n\in\mathcal{B}_b(\mathcal{A})$ for every $n\in\enne$. \andrea{Moreover, for any} $n \in \mathbb{N}$, the invariance of the measure $\gamma$ and the boundedness of $\Theta_n$ \andrea{imply}
\[
\int_{\mathcal{A}} \Theta_n(x)\: \gamma(\d x) = \int_{\mathcal{A}} P_t\Theta_n(x)\: \gamma(\d x)\qquad \forall \ t \ge 0.
\]
\andrea{By definition of $\Theta_n$, we also get}
\[
P_t \Theta_n(x):=\mathbb{E}\left[\Theta_n (u^x(t))\right]\le \mathbb{E}\left[ \|u^x(t)\|^2_H\right]
\]
\andrea{for all $t \geq 0$ and $n \in \enne$. Owing to Lemma \ref{lem:apriori1},} we have
\begin{multline}
\label{est:apriori1}
\E\|u(t)\|^2_H \le \|x\|^2_H\exp\left\{-\left(\frac{2}{K_2}\left[\min(\varepsilon,C_0)\mathbbm{1}_{L< +\infty}+ \left( \min(\varepsilon,C_0)-K_2\left(\frac{\widetilde{C}_G}{2}+ \widetilde{C}_J \right)\right)\mathbbm{1}_{L=+ \infty}\right] \right)t\right\}
\\
+\frac{C_1+\left(\frac{\widetilde{C}_G}{2}+ \widetilde{C}_J\right)(1+  \mathbbm{1}_{L<\infty}|\OO| L^2) }{\frac{1}{K_2}\left[\min(\varepsilon,C_0)\mathbbm{1}_{L< +\infty}+ \left( \min(\varepsilon,C_0)-K_2\left(\frac{\widetilde{C}_G}{2}+ \widetilde{C}_J \right)\right)\mathbbm{1}_{L=+ \infty}\right]}.
\end{multline}
\andrea{The additional assumption on the structural parameters yields that, by} letting $t \rightarrow + \infty$ in \eqref{est:apriori1}, the exponential term vanishes. \andrea{Letting $C >0$ denote the constant to the right hand side of \eqref{est:apriori1}, we therefore have}
\[
\limsup_{t \rightarrow + \infty } P_t \Theta_n(x) \le C,
\] 
\andrea{and, }as a consequence, 
\[
\int_{\mathcal{A}} \Theta_n(x)\:\gamma(\d x)\le C
\]
\andrea{for all $n\in \enne$. Since $C$ does not depend on $n$, the claim \eqref{est_inv_meas_H} is then a consequence of the monotone convergence theorem as $\Theta_n \to \Theta$ pointwise from below.}
Notice that the condition \[\min(\varepsilon,C_0)>K_2\left(\frac{\widetilde{C}_G}{2}+ \widetilde{C}_J \right)\] in the case $L=+\infty$ is needed \andrea{both} to ensure the existence of invariant measures and \andrea{to grant that} the constant $C$ in \eqref{est_inv_meas_H} is indeed positive. \\[0.3\baselineskip]
{\sc Step 2.} \andrea{In a similar fashion,} we now show that there exists a positive constant $C$ such that
\begin{equation}
\label{est_inv_meas_V}
\int_{\mathcal{A}}\|x\|_V^2\, \gamma({\rm d}x)\le C
\end{equation}
\andrea{for any invariant measure $\gamma$. Again, }we consider the mapping 
\[
\Lambda:\mathcal{A} \rightarrow [0,+\infty], \qquad
x \mapsto \|x\|^2_{V}\mathbbm{1}_{\mathcal{A} \cap V}(x) \qquad \forall\:x \in \A \cap V,
\]
\andrea{extended with value $+\infty$ outside $\A \cap V$,} and its approximati\andrea{ng sequence} $\{\Lambda_n\}_{n\in\enne}$, where for every $n\in\enne$,
\[
\Lambda_n:\mathcal{A}\to[0,n^2], \qquad 
\Lambda_n:x \mapsto 
\begin{cases}
\|x\|^2_{V} & \text{if} \ x \in B_n^V \cap \mathcal{A},
\\
n^2 & \text{otherwise},
\end{cases}
\quad x\in\mathcal{A},
\]
where $B_n^V$ is the closed ball of radius $n$ in $V$. \andrea{It is clear that} we have $\Lambda_n\in\mathcal{B}_b(\mathcal{A})$ for every $n\in\enne$. Exploiting the invariance of $\gamma$, the boundedness of $\Lambda_n$, the definition of $P$
and the Fubini--Tonelli theorem we have that
\begin{align*}
\int_{\mathcal{A}} \Lambda_n(x)\, \gamma(\d x) & = 
\int_0^1 \int_{\mathcal{A}} \Lambda_n(x) \, \gamma(\d x)\, {\rm d}s \\
& =
\int_0^1 \int_{\mathcal{A}} P_s \Lambda_n(x)\, \gamma(\d x)\, {\rm d}s\\
&= \int_0^1 \int_{\mathcal{A}} \E\left[\Lambda_n(u^x(s))\right]\, \gamma(\d x)\, {\rm d}s\\
&=\int_{\mathcal{A}}\int_0^1 \mathbb{E}\left[\Lambda_n(u^x(s)) \right]\,{\rm d}s\, \gamma(\d x)
\\
 &
\le
 \int_{\mathcal{A}}\int_0^1 \mathbb{E}\left[\|u^x(s)\|_V^2 \right]\,{\rm d}s\, \gamma(\d x).
\end{align*}
By Lemma~\ref{lem:apriori1} we infer that 
\begin{align*}
  \int_{\mathcal{A}}\int_0^1 \mathbb{E}\left[\|u^x(s)\|_V^2 \right]\,{\rm d}s\, \gamma(\d x)\le C \left(1 +\int_{\mathcal{A}}\|x\|_H^2\,\gamma(\d x) \right),
\end{align*}
\andrea{and therefore} estimate \eqref{est_inv_meas_H} yields 
\[
\int_{\mathcal{A}} \Lambda_n(x)\, \gamma(\d x) \le C,
\]
for a positive constant $C$ independent of $n \in \mathbb{N}$ \andrea{and for all $n \in \enne$}.
Since $\Lambda_n$ converges pointwise and monotonically from below to $\Lambda$, 
the monotone convergence theorem yields the second claim \eqref{est_inv_meas_V}.\\[0.3\baselineskip]
{\sc Step 3.} \andrea{A third approximation argument yields that} there exists a positive constant $C$ such that
\begin{equation}
\label{est_inv_meas_Psi}
\int_{\mathcal{A}} \|\Psi(x)\|_{L^1(\OO)}\: \gamma(\d x)\le C
\end{equation}
\andrea{for all invariant measures $\gamma$.}
Define the mapping
\[\Xi:\mathcal{A}\rightarrow [0, + \infty), \qquad 
x \mapsto \|\Psi(x)\|_{L^1(\OO)}, \qquad \forall \: x \in \mathcal{A}
\]
and its approximations $\{\Xi_n\}_{n\in\enne}$ given by
\begin{equation*}
\Xi_n:\A\to[0,n], \qquad x \mapsto 
\begin{cases}
\|\Psi(x)\|_{L^1(\OO)} & \text{if} \ \|\Psi(x)\|_{L^1(\OO)}\le n,\\
n & \text{otherwise},
\end{cases}
\qquad \forall \: x\in\A.
\end{equation*}
It holds that $\Xi_n\in\mathcal{B}_b(\mathcal{A})$ for every $n\in\enne$. Exploiting the invariance of $\gamma$, the boundedness of $\Xi_n$, the definition of $P$, the Fubini-Tonelli theorem and Lemma \ref{lem:apriori2} we infer 
\begin{align*}
    \int_{\mathcal{A}} \Xi_n(x)\: \gamma(\d x)
    &= \int_0^1\int_{\mathcal{A}} \Xi_n(x)\: \gamma(\d x)\,{\rm d}s
    \\
    &= \int_0^1\int_{\mathcal{A}} P_s\Xi_n(x)\: \gamma(\d x)\,{\rm d}s
    \\
    &=\int_{\mathcal{A}}\int_0^1 \mathbb{E}[\Xi_n(u^x(s))]\,{\rm d}s\: \gamma(\d x)
    \\
    &\le \int_{\mathcal{A}}\int_0^1 \mathbb{E}[\|\Psi(u^x(s))\|_{L^1(\OO)}]\,{\rm d}s\,\gamma(\d x)
    \\
    &\le C\left(1 + \int_{\mathcal{A}}\|x\|^2_H\, \gamma(\d x)\right).
\end{align*}
From \eqref{est_inv_meas_H} we thus infer the existence of a positive constant $C$, independent of $n \in \enne$, such that 
\[
 \int_{\mathcal{A}} \Xi_n(x) \: \gamma(\d x)\le C.
\]
As \andrea{in the previous steps}, since $\Xi_n$ converges pointwise and monotonically from below to $\Xi$, the monotone convergence theorem yields \eqref{est_inv_meas_Psi}.\\[0.3\baselineskip]
{\sc Step 4.} Eventually, \andrea{through a final iteration of an approximation argument,} we show that there exists a positive constant $C$ such that
\begin{equation}
    \label{est_inv_meas_Phi}
\int_{\mathcal{A}} \|\Phi'(x)\|^2_H\: \gamma(\d x)\le C.
\end{equation}
\andrea{for all invariant measures $\gamma$. Let}  
\[\Pi:\mathcal{A}\rightarrow [0, + \infty], \qquad 
x \mapsto 
\begin{cases}
  \|\Phi'(x)\|^2_{H} \quad&\text{if } \Phi'(x) \in H,\\
  +\infty \quad&\text{otherwise},
\end{cases}
\qquad \forall \: x \in \mathcal{A},
\]
and its approximations $\{\Pi_n\}_{n\in\enne}$, defined for every $n\in\enne$ \andrea{by}
\begin{equation*}
\Pi_n:\A\to[0,n^2], \qquad \Pi_n:x \mapsto 
\begin{cases}
\|\Phi'(x)\|^2_{H} & \text{if} \ \|\Phi'(x)\|_H\le n,\\
n^2 & \text{otherwise},
\end{cases}
\quad x\in\A.
\end{equation*}
\andrea{Once again,} it holds that $\Pi_n\in\mathcal{B}_b(\A)$ for every $n\in\enne$. Arguing similarly as above, thanks to Lemma \ref{lem:apriori2}, we get 
\begin{align*}
    \int_{\mathcal{A}} \Pi_n(x)\: \gamma(\d x)
    &= \int_0^1\int_{\mathcal{A}} \Pi_n(x)\: \gamma(\d x)\,{\rm d}s
    \\
    &= \int_0^1\int_{\mathcal{A}} P_s\Pi_n(x)\: \gamma(\d x)\,{\rm d}s
    \\
    &=\int_{\mathcal{A}}\int_0^1 \mathbb{E}[\Pi_n(u^x(s))]\,{\rm d}s\: \gamma(\d x)
    \\
    &\le \int_{\mathcal{A}}\int_0^1 \mathbb{E}[\|\Phi'(u^x(s))\|^2_H]\,{\rm d}s\: \gamma(\d x)
    \\
    &\le C\left(1 + \int_{\mathcal{A}}\left(\|x\|^2_H +\|\Psi(x)\|_{L^1(\OO)}\right)\: \gamma(\d x)\right).
\end{align*}
\andrea{It follows then from} \eqref{est_inv_meas_H} and \eqref{est_inv_meas_Psi} the existence of a positive constant $C$, independent of $n \in \enne$, such that 
\[
 \int_{\mathcal{A}} \Pi_n(x)\: \gamma(\d x)\le C
\]
\andrea{for all $n \in \mathbb N$.} Since $\Pi_n$ converges pointwise and monotonically from below to $\Pi$, the monotone convergence theorem yields \eqref{est_inv_meas_Phi}. \andrea{On account of \eqref{est_inv_meas_V} and \eqref{est_inv_meas_Phi}, we conclude that $\gamma$ is indeed supported on $\mathcal A_{\mathrm{str}}$.}
\subsection{Existence of an ergodic invariant measure}
Let us recall first the definition of ergodicity for the transition semigroup $P$.
In this direction, note that for every invariant measure $\gamma$,
by density and by definition of invariance,
the semigroup $P$ can be extended (with the same symbol for \andrea{convenience}) to 
a strongly continuous linear semigroup of contractions on $L^p(\mathcal{A},\mu)$ for every $p\in[1,+\infty)$.

\begin{defin}
  \label{def:erg}
  An invariant measure $\gamma\in\cP(\mathcal{A})$ for the semigroup $P$ is said to be ergodic if
  \[
  \lim_{t \rightarrow \infty}\frac 1t \int_0^tP_s \varphi\, {\rm d}s 
  = \int_{\mathcal{A}}\varphi(x) \, \gamma(\d x) \qquad \text{in} \ L^2(\mathcal{A}, \gamma)\quad \forall\, \varphi \in L^2(\mathcal{A}, \gamma).
  \]
\end{defin} \noindent
The proof of existence of ergodic invariant measure is analogous to the \andrea{one} of \cite[Proposition 3.7]{SZ}, \andrea{and therefore we omit the details}.
\subsection{Uniqueness of the invariant measure}
\andrea{Finally,} we show that, for Dirichlet boundary conditions, provided the system is sufficiently dissipative, the invariant measure is unique, strongly and exponential\andrea{ly} mixing, according to the following definition.
\begin{defin}
  \label{def:st_mix}
  An invariant measure $\gamma\in\cP(\mathcal{A})$ for the semigroup $P$ is said to be strongly mixing if
  \[
  \lim_{t \rightarrow +\infty}P_t \varphi
  = \int_{\mathcal{A}}\varphi(x) \, \gamma(\d x) \quad 
  \text{in} \ L^2(\mathcal{A}, \gamma)\qquad \forall\, \varphi \in L^2(\mathcal{A}, \gamma),
  \]
  and it is said to be exponential\andrea{ly} mixing if for any $\varphi \in \text{Lip}(\mathcal{A})$ 
  \[
  \left|P_t \varphi(x)-\int_{\mathcal{A}}\varphi(y)\,\gamma(\d y)\right|\le C(x)\|\varphi\|_{\text{Lip}}e^{-c t}, \qquad  \forall \:x \in \mathcal{A}, \quad t  > 0,
  \]
  for some positive constant $c$ and function $C(x)$. 
\end{defin} \noindent
\andrea{In order to prove uniqueness of invariant measures, we proceed as follows, recalling that we ask for the additional assumption
\[
\varepsilon \rho_1> \frac{C_G}{2} + C_J+ \sqrt{C_F} \qquad  \text{if} \ L<+\infty,
\]
or
\[
\min\left(\varepsilon \rho_1, \min(\varepsilon,C_0)\right)>\max\left(K_2\left(\frac{\widetilde{C}_G}{2}+ \widetilde{C}_J \right), \frac{C_G}{2}+C_J+\sqrt{C_F}\right) \qquad \text{if} \ L=+ \infty.
\]}Let $x, y \in \mathcal{A}$ and let $u^x$, $u^y$ be the \andrea{corresponding} solutions \andrea{admitting $x$ and $y$ as initial conditions, respectively}. Setting $w:=u^x-u^y$, the It\^o formula for $\|w\|^2_H$ yields, $\mathbb{P}$-almost surely, for every $t \ge 0$,
\begin{multline*}
     \dfrac{1}{2}\|w(t)\|^2_H + \varepsilon\int_0^t \|\nabla w(s)\|^2_{\b H} \: \d s \\  = \frac12\|x-y\|^2_H- \int_0^t \left(w(s), \Psi'(u^x(s))-\Psi'(u^y(s))\right)_H \: \d s - \int_0^t \left(w(s), F(u^x(s))-F(u^y(s))\right)_H \: \d s \\
     + \int_0^t \left(w(s), (G(u^x(s))-G(u^y(s))) \: \d W(s) \right)_H  + \dfrac{1}{2} \int_0^t \|G(u^x(s))-G(u^y(s))\|^2_{\mathscr L^2(U, H)} \: \d s \\
     + \int_0^t \int_Z (w(s^-), J(u^x(s^-),z) - J(u^y(s^-),z))_H\: \overline{\mu}(\d s,\, \d z) \\
    + \int_0^t \int_Z \|J(u^x(s^-),z) - J(u^y(s^-),z)\|^2_H \: \nu(\d z) \: \d s.
\end{multline*}
Reasoning as in Subsection \ref{ssec:uniq}, exploiting the Poincar\'e inequality and bearing in mind that the stochastic integrals are martingales, we obtain the estimate 
\[
\mathbb{E}\left[ \|w(t)\|^2_H\right] + 2\alpha\mathbb{E}\int_0^t \| w(s)\|^2_H \: \d s \le  \|x-y\|^2_H,
\]
where $\alpha:=\varepsilon\rho_1 - \left(\frac{C_G}{2}+C_J+\sqrt{C_F} \right)$ is a positive constant by assumption. By the Gronwall lemma we obtain 
\begin{equation}\label{uniq3}
  \E\|(u^x-u^y)(t)\|^2_H \le e^{-2\alpha t} \|x-y\|^2_H,
\quad\forall\,t\ge0.
\end{equation}
Consequently, let $\gamma$ be an invariant measure for \andrea{the transition semigroup}
$P$. For any $\varphi \in \mathcal C_b^1(H)$ and 
$x \in \mathcal{A}$, by definition of invariance 
and the estimate \eqref{uniq3} we have
\begin{align}
\label{strongly_mixing}
\left| P_t(\varphi|_{\mathcal{A}})(x)-\int_{\mathcal{A}}\varphi(y)\, \gamma({\rm d}y)\right|^2 
&\le \|D\varphi\|^2_{\mathcal C_b(H)}\int_{\mathcal{A}}\mathbb{E}\left[ \|u^x(t)-u^y(t)\|^2_H\right]\, \gamma({\rm d}y)
\notag\\
&\le \|D\varphi\|^2_{\mathcal C_b(H)}e^{-2\alpha t}\int_\mathcal{A} \|x-y\|_H^2\: \gamma (\d y),
\end{align}
\andrea{and the right hand side} converges to zero as $t \rightarrow + \infty$ since \[\int_\mathcal{A} \|y\|^2_H\: \gamma(\d y)< + \infty\]
by virtue of \andrea{the argument shown in Subsection \ref{ssec:inv_support}}. Since $C^1_b(H)|_{\mathcal{A}}$ is dense in $L^2(\mathcal{A}, \gamma)$, we infer that 
\begin{equation*}
\left| P_t\varphi(x)-\int_{\mathcal{A}}\varphi(y)\: \gamma(\d y)\right|^2\rightarrow 0
 \quad \text{as} \ t \rightarrow +\infty, \quad \forall \: \varphi \in L^2(\mathcal{A}, \gamma),
\end{equation*}
from which we deduce that the strong mixing property holds true, again by the arguments in \andrea{Subsection \ref{ssec:inv_support}}. By computations similar to \eqref{strongly_mixing} and \andrea{the argument shown in Subsection \ref{ssec:inv_support}}, for any $\varphi\in \text{Lip}(\mathcal{A})$, we also infer
\[
\left|P_t \varphi(x)-\int_\mathcal{A}\varphi(y)\: \gamma(\d y)\right|\le C(1+ \|x\|_H)\|\varphi\|_{\text{Lip}}e^{-\alpha t}, \qquad \forall \: x \in \mathcal{A}, \quad t >0.
\]
for a positive constant $C$, that is, the exponential mixing property holds true. The above computation easily implies also the uniqueness of the invariant measure. 
Indeed, let $\pi$ be another invariant measure, then for all $\varphi \in \mathcal C_b^1(H)$ we have
\begin{align*}
\left|  \int_{\mathcal{A}}\varphi(y)\, \gamma({\rm d}y)- \int_\mathcal{A}\varphi(x)\, \pi({\rm d}x) \right| 
&=\left|  \int_{\mathcal{A} }\int_{\mathcal{A}}\left(P_t(\varphi|_{\mathcal{A}})(y)
-P_t(\varphi|_{\mathcal{A}})(x)\right)\,  \pi({\rm d}x) \: \gamma({\rm d}y) \right|
\\
&\le \|D\varphi\|^2_{\mathcal C_b(H)} e^{-2\alpha t}  \int_{\mathcal{A} }\int_{\mathcal{A}}\|x-y\|^2_H
\,  \pi({\rm d}x) \:\gamma({\rm d}y) \rightarrow 0 \qquad \text{as} \  t \rightarrow +\infty.
\end{align*}
The fact that the unique invariant measure is also ergodic follows from the existence of ergodic invariant measures, and this concludes the proof.
\section{Approximation and simulation}
\label{sec:numerics}
\subsection{Approximation}

We approximate \eqref{eq:strongform} by a semi-implicit Euler--Maruyama method in time and a conforming finite element discretization in space. 
The Laplacian and the singular monotone part of the drift are treated implicitly, while the stochastic increments are handled explicitly. 
This is natural in the presence of singular convex potentials, where an implicit step prevents spurious excursions outside the physically admissible range and mirrors the monotonicity structure exploited in the analysis. 
In contrast, for polynomial double-well potentials the drift may grow superlinearly, and standard explicit schemes may diverge \cite{beccari2019strong,BG19}; in that setting, tamed variants are often used \cite{kruse2023bdf2,huang2023stability,fritz2025analysis}. 
Here, since the constraint is enforced through a singular convex contribution, a semi-implicit treatment is the appropriate stabilization mechanism. 

\subsubsection*{Time discretization} Let $0=t_0<t_1<\dots<t_N=T$ with $t_{n+1}-t_n=\tau$ for all $n \in \{0, .., N-1\}$. 
Following the previous discussion, we discretize \eqref{eq:strongform} by the implicit--explicit Euler--Maruyama scheme
\begin{equation}\label{eq:scheme-strong}
u^{n+1} - \tau\varepsilon \Delta u^{n+1} + \tau \Psi_\lambda'(u^{n+1}) + \tau F_1(u^{n+1})
= u^n + \tau F_2(u^n) +G(u^n)\,\Delta W^n + \Delta J^n,
\end{equation}
where $F(u^{n+1},u^n)=F_1(u^{n+1})+F_2(u^n)$ is divided into its implicit and explicit components, $\Delta W^n=W(t_{n+1})-W(t_n)$ and $\Delta J^n$ is a compensated jump increment.

In the following, we discuss each component:
\begin{itemize}
\item We split the potential function as in the analysis: $\partial\Psi+\!F_1$, with $\Psi$ convex and singular at the endpoints. 
 Concretely, we take as convex singular contribution the logarithmic entropy
\[
\Psi(u)=\theta\Big(u\ln u+(1-u)\ln(1-u)\Big),
\qquad u\in(0,1),
\]
which yields a strong barrier at $u=0$ and $u=1$, while the perturbation $F_1(u)=-4\theta_0(u-\frac12)$ is Lipschitz continuous and fits assumption~\ref{hyp:F}, with $\theta<\theta_0$ as specified below. 
\item The cylindrical Wiener term models continuous background fluctuations. 
At each time step we generate a mean-zero spatial random field by a truncated sine expansion,
\[
\eta^n(x)=\alpha \sum_{k=1}^{K}\sum_{l=1}^{L}\xi_{k,l}^n\,\sin(k\pi x_1)\sin(l\pi x_2),
\qquad \xi_{k,l}^n\overset{\text{i.i.d.}}{\sim}\mathcal N(0,1),
\]
and enforce $\int_\Omega \eta^n\,\mathrm{d}x=0$ by subtracting its spatial mean. 
The multiplicative prefactor $G(u)=c_{\mathrm{noise}}\,u(1-u)$ is chosen interface-localized 
so that noise vanishes in the pure phases.
    \item The jump term represents abrupt localized damage events and is modeled by a (possibly compensated) Poisson integral.
In the compensated case we discretize the increment
\[
\int_{t_n}^{t_{n+1}}\!\int_Z J(u^-,z)\,\bar\mu(\mathrm{d}t,\mathrm{d}z),
\qquad \bar\mu=\mu-\nu(\mathrm{d}z)\,\mathrm{d}t,
\]
by the standard Euler--Maruyama approximation
\begin{equation}\label{eq:comp-jump-increment}
\Delta J^n
=\sum_{k=1}^{N_J^n} J(u^n,z_k^n)\;-\;\tau\!\int_Z J(u^n,z)\,\nu(\mathrm{d}z),
\end{equation}
where $N_J^n\sim\mathrm{Poisson}(\tau\,\nu(Z))$ and the marks $z_k^n$ are sampled i.i.d.\ from $\nu/\nu(Z)$.
By construction, $\mathbb E[\Delta J^n\,|\,\mathcal F_{t_n}]=0$. 
In an {uncompensated} model one instead replaces $\bar\mu$ by $\mu$ (equivalently, drops the drift correction term in \eqref{eq:comp-jump-increment} or chooses $F_2$ as the compensator); this corresponds to a systematic accumulation of damage.
Each event is centered at a random location $z\in\OO$ and acts through a smooth track kernel,
\begin{equation}\label{eq:J-choice}
J(u,z)(x) = A_{\mathrm{jump}}(u(x))\,\kappa(x;z),
\qquad 
\kappa(x;z)=\exp\!\Big(-\tfrac{\|x-z\|^2}{2\sigma_{\mathrm{track}}^2}\Big).
\end{equation}
The choice of the function $A_{\mathrm{jump}}$ ensures that jumps are localized to intermediate states and negligible in the pure phases. 
Since $0\le \kappa\le 1$, this fits the uniform boundedness assumptions on $J$ used in the analysis.
\end{itemize}

\subsubsection*{Space-time discretization}
Let $\mathcal{T}_h$ be a triangulation of $\OO$ with mesh size $h$ and $V_h\subset H^1(\Omega)$ the space of continuous piecewise-linear functions. 
The fully discrete problem reads: find $u_h^{n+1}\in V_h$ such that for all $v_h\in V_h$,
\begin{multline}\label{Eq:ACfulldiscrete}
    (u_h^{n+1},v_h) + \tau \varepsilon (\nabla u_h^{n+1},\nabla v_h)
+ \tau (\Psi_\lambda'(u_h^{n+1}),v_h) + \tau (F_1(u_h^{n+1}),v_h) \\
= (u_h^n,v_h)+ \tau (F_2(u_h^{n}),v_h) + (G(u_h^n)\Delta W^n,v_h)
+ (\Delta J_h^n,v_h),
\end{multline} 
where $\Delta J_h^n$ is the finite element representation of \eqref{eq:comp-jump-increment}.
At each time step we solve the resulting nonlinear system by Newton's method. 
The implementation is carried out in the Firedrake framework \cite{rathgeber2016firedrake}.

\subsection{Simulation setup}

We solve the stochastic Allen--Cahn model with the semi-implicit variational scheme from \eqref{Eq:ACfulldiscrete} to illustrate mesoscopic dynamics under continuous fluctuations and sudden, localized damage hits. 
In this numerical section the phase field $u(t,x)\in[0,1]$ denotes the local damage level, with $u=0$ corresponding to an undamaged (healthy) state and $u=1$ to fully damaged, irreparable material. 
We investigate two different cases under different regimes by tuning the jump intensity $\lambda_{\mathrm{jump}}$, which is related to the expected number of radiation events per unit area and time. We refer to the four cases of $\lambda_{\mathrm{jump}} \in \{0,10,50,100\}$ as "None", "Few", "Some", and "Many", respectively,

The fixed parameters are $\tau=0.05$, $T=10$, $h=1/128$, $\OO=(0,1)^2$, $\varepsilon=1/1600$, $\theta=1/2$, $\theta_0=1$, and $\sigma_{\mathrm{track}}=1/10$. 
For the Poisson jump counter $N^n\sim\mathrm{Poisson}(\xi)$ per time step we have
\[
\xi=\mathbb E[N^n]=\lambda_{\mathrm{jump}}\,\tau\,|\OO|,
\qquad 
\mathbb P(N^n=k)=e^{-\xi}\frac{\xi^k}{k!}.
\]
In particular (since $|\OO|=1$), $\xi=0.05\,\lambda_{\mathrm{jump}}$, so for $\lambda_{\mathrm{jump}}=10$ one has $\mathbb P(N^n=0)=e^{-0.5}\approx0.61$, whereas for $\lambda_{\mathrm{jump}}\ge 50$ jumps occur in almost every step ($\mathbb P(N^n=0)=e^{-2.5}\approx 0.082$ for $\lambda_{\mathrm{jump}}=50$).

\subsection{Case 1: Random initial datum and compensated jump}

To emphasize damage nucleation from a nearly healthy configuration, we initialize the system close to $u\equiv 0.5$ with a small smooth random perturbation:
\begin{equation}\label{eq:init-random}
u(x,0)=0.5+\eta(x),
\qquad x\in\OO,
\end{equation}
where $\eta$ is a smooth mean-zero random field generated by a truncated sine series and scaled to have small amplitude. We choose $c_{\mathrm{noise}}=1/2$.
The jump forcing is implemented in compensated form, consistently with \eqref{eq:comp-jump-increment}, and we choose $A_\mathrm{jump}(u)=\frac12 u(1-u)$. Therefore, varying the jump parameter $\lambda_\mathrm{jump}$ primarily modifies the intermittency, spatial clustering and variance of damage patterns. 

Figure~\ref{fig:case1_evo} illustrates the evolution of the damage field for increasing jump intensities. 
Starting from a small random perturbation around $u\equiv 0.5$, the dynamics rapidly separates into near-pure phases, and the spatial organization at the final time depends visibly on the jump activity: higher jump frequencies lead to a denser pattern of localized nucleation events and thus to a different coarsening morphology.
The rightmost column shows the sampled jump locations, highlighting the increasingly clustered track pattern as $\lambda_{\mathrm{jump}}$ grows.
Figure~\ref{fig:case1_minmax} reports the evolution of the total damage $\int_\OO u(t)\,\mathrm{d}x$ (ensemble mean with uncertainty bands). 
In the compensated setting, the total damage remains approximately constant in the regimes ``None'' and ``Many'', while the intermediate regime ``Few'' may display a more noticeable deviation: when jumps are rare, single realizations can produce a stronger intermittent bias, and the ensemble is more sensitive to finite-sample fluctuations.
Finally, Figure~\ref{fig:case1_minmax} also shows that $u_{\min}(t)$ and $u_{\max}(t)$ stay strictly inside $(0,1)$ for all jump intensities, confirming that the logarithmic barrier effectively prevents excursions to the endpoints at the discrete level.

\subsection{Case 2: Circular initial datum and uncompensated jump}

To model a single, initially healthy cell nucleus, we initialize the viability field $u(x, 0)$ as a circular region of healthy tissue ($u = 0$) with radius $0.4$, surrounded by an extracellular environment considered lethally damaged ($u = 1$). A smooth hyperbolic tangent transition zone with thickness proportional to $\varepsilon$ is used to regularize the interface:
\[
u(x, 0) = \frac{1}{2} \left(1 - \tanh\!\left( \frac{0.4 - \|x - (\tfrac12,\tfrac12)\|}{\sqrt{2\varepsilon}} \right) \right).
\]
This simplified geometry allows for a clear analysis of radiation damage within a well-defined biological unit. This time, we choose $c_{\mathrm{noise}}\in\{\frac12,5\}$ and investigate the influence of the increase in the random fluctuations due to the Wiener noise. Moreover, we use an uncompensated jump by choosing $F_2$ as the compensator and we choose $A_\mathrm{jump}=\frac12 (1-u)$ to allow damage events even in the region with $u=0$, so fully healthy regions may be damaged.

In this interpretation, the uncompensated jump forcing describes abrupt, spatially localized radiation hits (tracks) that
increase damage on average, whereas the Wiener perturbation represents persistent micro-scale variability.
This perspective is in line with phase-field modeling approaches in oncology, where therapeutic interventions
(chemotherapy, radiotherapy, immunotherapy) are incorporated as additional forcing terms and one is interested not only
in tumor control but also in the impact on neighboring healthy structures; see, e.g., the discussion of treatment effects in
\cite[Sec.~2.7]{fritz2023tumor}.
In particular, the track kernel in \eqref{eq:J-choice} provides a simple mechanism to model the fact that radiation is not perfectly focused:
even when the target is localized, stochastic energy deposition may induce collateral damage in the surrounding region.

\begin{figure}[H] \centering
\includegraphics[width=.9\textwidth,page=1]{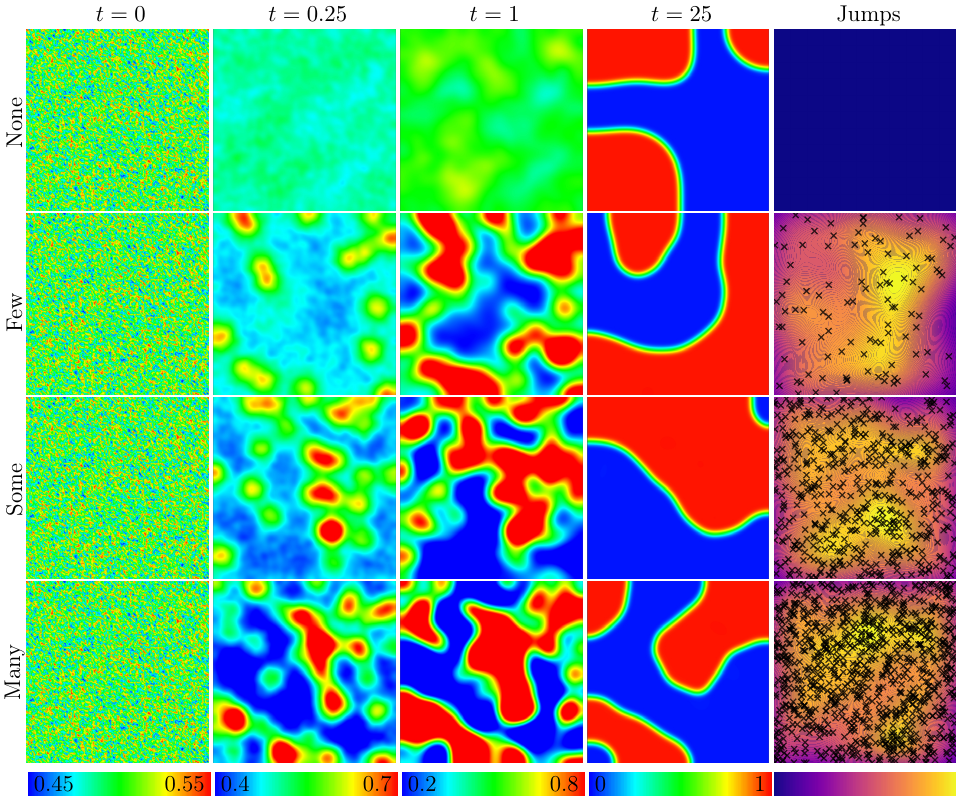}
\caption{Case 1 (random initial condition, compensated jumps): snapshots of $u(t,\cdot)$ at $t=0,0.25,1,25$ for increasing jump intensities (None/Few/Some/Many). The last column visualizes the jump locations up to final time.}
\label{fig:case1_evo}
\end{figure} \vspace{-.3cm}

\begin{figure}[H]
\centering
\includegraphics[width=.9\textwidth,page=2]{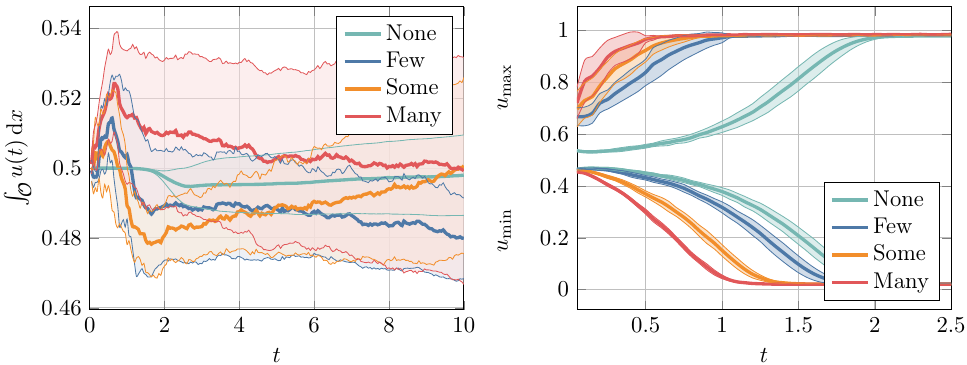}
\caption{Case 1 (compensated jumps): left: evolution of the total damage $\int_\OO u(t)\,\mathrm{d}x$ for different jump intensities (ensemble mean with uncertainty bands); right: evolution of $u_{\min}(t)$ and $u_{\max}(t)$ for different jump intensities. The values remain strictly within $(0,1)$, indicating effective enforcement of the logarithmic barrier.}
    \label{fig:case1_minmax}
\end{figure}

\begin{figure}[H] \centering
\includegraphics[width=.9\textwidth,page=3]{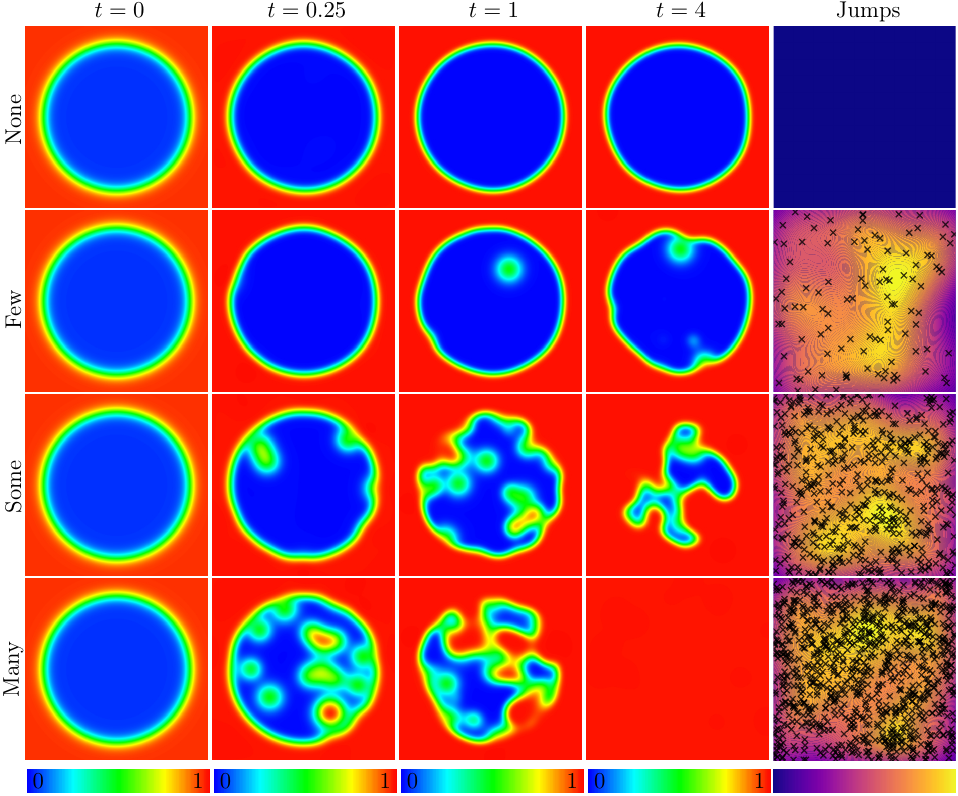}
\caption{Case 2 (circular initial condition, {uncompensated} jumps): snapshots of $u(t,\cdot)$ at $t=0,0.25,1,4$ for increasing jump intensities (None/Few/Some/Many), $c_\mathrm{noise}=1/2$. The last column visualizes the jump locations up to final time.}
\label{fig:case2_evo}
\end{figure}

\begin{figure}[H]
    \centering
\includegraphics[width=.9\textwidth,page=4]{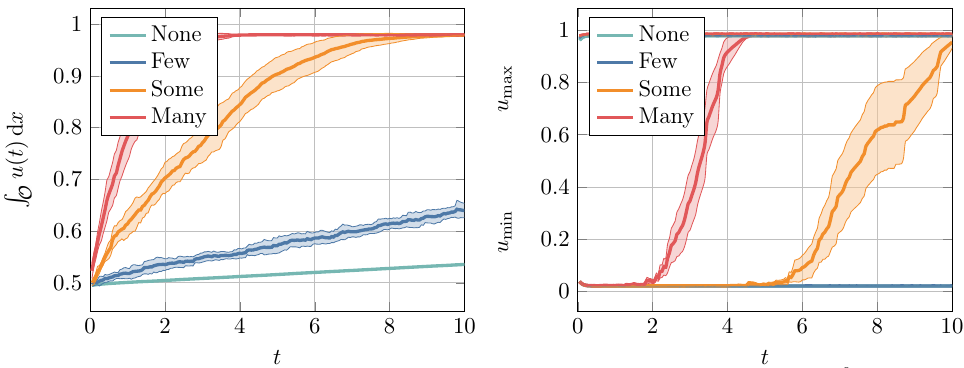}
\caption{Case 2 (uncompensated jumps): left: evolution of the total damage $\int_\OO u(t)\,\mathrm{d}x$ for different jump intensities (ensemble mean with uncertainty bands). Increasing jump intensity accelerates the growth of total damage; right: evolution of $u_{\min}(t)$ and $u_{\max}(t)$ for different jump intensities. The solution remains strictly within $(0,1)$ for all regimes.}
\label{fig:case2_minmax}
\end{figure}

Figure~\ref{fig:case2_evo} shows the evolution of a healthy circular region embedded in a damaged background, choosing $c_\mathrm{noise}=1/2$
In contrast to Case~1, the jump term is uncompensated and thus injects damage on average.
Consequently, the total damage $\int_\OO u(t)\,\mathrm{d}x$ increases monotonically in Figure~\ref{fig:case2_minmax}, and the destruction of the healthy region accelerates as the jump intensity increases: for ``Many'' jumps the cell is rapidly eroded and the domain becomes fully damaged.
The corresponding min/max curves in Figure~\ref{fig:case2_minmax} again remain strictly within $(0,1)$, indicating robust enforcement of the physical bounds across all regimes. Lastly, we depict in Figure~\ref{fig:case2high_evo} the case of $c_\mathrm{noise}=5$. It is visible that the interface is deformed in contrast to the case of $c_\mathrm{noise}=1/2$ from before.

\begin{figure}[H] \centering
\includegraphics[width=.9\textwidth,page=5]{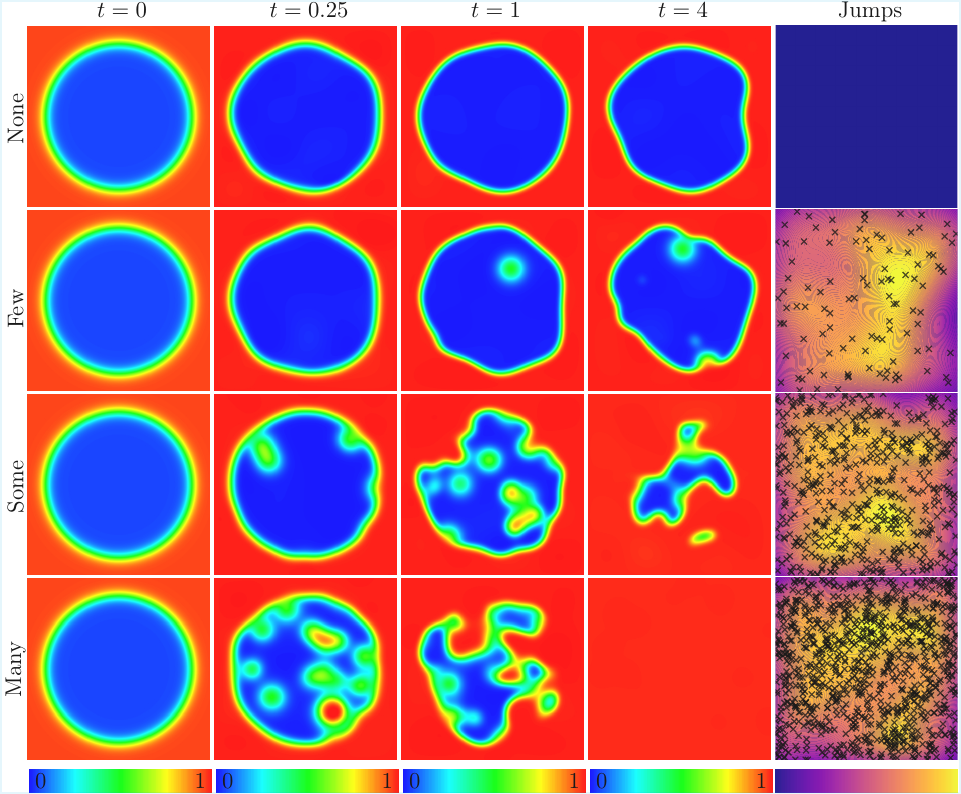}
\caption{Case 2 (circular initial condition, {uncompensated} jumps): snapshots of $u(t,\cdot)$ at $t=0,0.25,1,4$ for increasing jump intensities (None/Few/Some/Many), with $c_\mathrm{noise}=5$. The last column visualizes the jump locations up to final time.}
\label{fig:case2high_evo}
\end{figure}

\appendix
\section{Further useful estimates}
\label{App_A}
\andrea{In this appendix, we include two useful estimates that are needed in the investigation of the longtime behavior of the system, i.e., in Section \ref{sec:invariant}. Let $x\in H$ be such that $\Psi(x) \in L^1(\OO)$ and} let $u$ be the unique corresponding solution to \ref{eq:strongform} given \andrea{by} Theorem \ref{thm:wellposed}. Define also the functional 
\[
\mathcal{H}_0: D(\mathcal H_0) \subset V \rightarrow \mathbb{R}, \qquad \mathcal{H}_0(v)=\int_{\OO}\Phi(v) \: \andrea{\d y},
\]
with $\Phi$ given in Lemma \ref{rem:Phi}. \andrea{The first estimate reads as follows.}
\begin{lem}
\label{lem:apriori1}
    Let Assumptions \ref{hyp:structural}-\ref{hyp:J} hold.   
    Then, for every $x\in\mathcal A$ and $t \geq 0$ it holds that 
  \begin{multline*}
\E\|u(t)\|^2_H \le \|x\|^2_H\exp\left\{-\left(\frac{2}{K_2}\left[\min(\varepsilon,C_0)\mathbbm{1}_{L< +\infty}+ \left( \min(\varepsilon,C_0)-K_2\left(\frac{\widetilde{C}_G}{2}+ \widetilde{C}_J \right)\right)\mathbbm{1}_{L=+ \infty}\right] \right)t\right\} 
\\
+\frac{C_1+\left(\frac{\widetilde{C}_G}{2}+ \widetilde{C}_J\right)(1+  \mathbbm{1}_{L<\infty}|\OO| L^2) }{\frac{1}{K_2}\left[\min(\varepsilon,C_0)\mathbbm{1}_{L< +\infty}+ \left( \min(\varepsilon,C_0)-K_2\left(\frac{\widetilde{C}_G}{2}+ \widetilde{C}_J \right)\right)\mathbbm{1}_{L=+ \infty}\right]}
\end{multline*}
    and 
    \begin{multline*}
    \left[\min(\varepsilon,C_0)\mathbbm{1}_{L< +\infty}+ \left( \min(\varepsilon,C_0)-K_2\left(\frac{\widetilde{C}_G}{2}+ \widetilde{C}_J \right)\right)\mathbbm{1}_{L=+ \infty}\right] \E\int_0^t \| u(s)\|^2_V \: \d s  
     \\
     \le \frac 12\|x\|^2_H  +  \left(C_1+\left(\frac{\widetilde{C}_G}{2}+ \widetilde{C}_J\right)(1+  \mathbbm{1}_{L<\infty}|\OO| L^2) \right)t.
    \end{multline*}
    \end{lem}

\begin{proof}
We apply the It\^o formula to the squared $H$-norm \andrea{of the solution $u$ at time $t$}. For every $t \ge 0$, it holds 
\begin{multline} 
\label{Ito_H_2}
     \dfrac{1}{2}\|u(t)\|^2_H + \varepsilon \int_0^t \|\nabla u(s)\|^2_{\b H} \: \d s + \int_0^t \left(u(s), \Phi'(u(s))\right)_H \: \d s = \frac 12\|x\|^2_H \\
     + \int_0^t \left(u(s), G(u(s))\: \d W(s) \right)_H  + \dfrac{1}{2} \int_0^t \|G(u(s))\|^2_{\mathscr L^2(U, H)} \: \d s \\
     + \int_0^t \int_Z (u(s^-), J(u(s^-),z))_H\: \overline{\mu}(\d s,\, \d z) 
     + \int_0^t \int_Z \|J(u(s^-), z)\|^2_H \: \nu(\d z) \: \d s,
\end{multline}
\andrea{$\P$-almost surely}. For every $t \ge 0$ from Remark \ref{rem:Phi} we infer 
\[
 \int_0^t \left(u(s), \Phi'(u(s))\right)_H \: \d s \ge C_0 \int_0^t \|u(s)\|^2_H -C_1 t,
\]
whereas Remark \ref{rem:linear_growth} yields
\begin{multline}
\label{stima_int}
\dfrac{1}{2} \int_0^t \|G(u(s))\|^2_{\mathscr L^2(U, H)} \: \d s
     + \int_0^t \int_Z \|J(u(s^-), z)\|^2_H \: \nu(\d z) \: \d s
     \\
     \le \left(\frac{\widetilde{C}_G}{2}+ \widetilde{C}_J\right)\left(t(1+ \mathbbm{1}_{L<\infty}|\OO| L^2) + \mathbbm{1}_{L=+ \infty} \int_0^t\|u(s)\|^2_H\, {\rm d}s \right).
\end{multline}
Therefore, we obtain the estimate 
\begin{multline*} 
     \dfrac{1}{2}\|u(t)\|^2_H + \varepsilon \int_0^t \|\nabla u(s)\|^2_{\b H} \: \d s +C_0\int_0^t \|u(s)\|^2_H \: \d s 
     \\
     \le \frac 12\|x\|^2_H  +  \left(C_1+\left(\frac{\widetilde{C}_G}{2}+ \widetilde{C}_J\right)(1+  \mathbbm{1}_{L<\infty}|\OO| L^2) \right)t +  \left(\frac{\widetilde{C}_G}{2}+ \widetilde{C}_J \right)\mathbbm{1}_{L=+ \infty} \int_0^t\|u(s)\|^2_H\, {\rm d}s 
      \\
      + \int_0^t \left(u(s), G(u(s))\: \d W(s) \right)_H  
     + \int_0^t \int_Z (u(s^-), J(u(s^-),z))_H\: \overline{\mu}(\d s,\, \d z).
\end{multline*}
Taking \andrea{expectations} on both sides of the above inequality and bearing in mind that the stochastic integrals are martingales, we get 
\begin{multline*} 
     \dfrac{1}{2}\E\|u(t)\|^2_H + \left[\min(\varepsilon,C_0)\mathbbm{1}_{L< +\infty}+ \left( \min(\varepsilon,C_0)-K_2\left(\frac{\widetilde{C}_G}{2}+ \widetilde{C}_J \right)\right)\mathbbm{1}_{L=+ \infty}\right] \E\int_0^t \| u(s)\|^2_V \: \d s  
     \\
     \le \frac 12\|x\|^2_H  +  \left(C_1+\left(\frac{\widetilde{C}_G}{2}+ \widetilde{C}_J\right)(1+  \mathbbm{1}_{L<\infty}|\OO| L^2) \right)t,
\end{multline*}
which proves the second statement. By the continuous Sobolev embedding $V \embed H$, from the above estimate we infer 
\begin{multline*} 
     \dfrac{1}{2}\E\|u(t)\|^2_H + \frac{1}{K_2}\left[\min(\varepsilon,C_0)\mathbbm{1}_{L< +\infty}+ \left( \min(\varepsilon,C_0)-K_2\left(\frac{\widetilde{C}_G}{2}+ \widetilde{C}_J \right)\right)\mathbbm{1}_{L=+ \infty}\right] \E\int_0^t \| u(s)\|^2_H \: \d s  
     \\
     \le \frac 12\|x\|^2_H  +  \left(C_1+\left(\frac{\widetilde{C}_G}{2}+ \widetilde{C}_J\right)(1+  \mathbbm{1}_{L<\infty}|\OO| L^2) \right)t
\end{multline*}
and the Gronwall lemma yields 
the first claim. This concludes the proof.
\end{proof} \noindent
\andrea{The second estimate concerns integrability properties of the possibly singular nonlinearity of the problem.}
\begin{lem}
\label{lem:apriori2}
    Let Assumptions \ref{hyp:structural}-\ref{hyp:J} hold.  
    Then, for every $t \ge 0$ there exists 
    a constant $C>0$ such that, for all $x\in\mathcal A$,
  \begin{equation*}
  \mathbb{E} \int_0^t \|\Psi(u(s))\|_{L^1(\OO)}\:{\rm d}s \le C(1 + \|x\|^2_H)
  \end{equation*}
  and 
   \begin{equation*}
  \mathbb{E} \int_0^t \|\Phi'(u(s))\|_{H}^2\,{\rm d}s \le C(1+ \|x\|^2_H+ \|\Psi(x)\|_{L^1(\OO)}),
  \end{equation*}
\end{lem}
\begin{proof}
    To prove \andrea{the first claim}, we start from \eqref{Ito_H_2}. \andrea{For any function $f: \mathbb R \to \mathbb R$, let $f^*$ denote its convex conjugate. Recalling that for any $r$ and $s \in \mathbb{R}$, the Fenchel--Young inequality states that 
    \[
    rs=\Psi(r)+\Psi^*(s) \Leftrightarrow s \in \partial \Psi(r),
    \]
    choosing $s = \Psi'(r)$ yields $r\Psi'(r)=\Psi(r)+\Psi^*(\Psi'(r))$ for all $r \in D(\Psi')$, since in our case the subdifferential is single-valued.}
   Using the fact that the stochastic integrals are martingales, thanks to estimate \eqref{stima_int} and Assumption \ref{hyp:F}, we thus infer 
   \begin{multline*}
         \dfrac{1}{2}\mathbb{E}\|u(t)\|^2_H + \varepsilon \mathbb{E} \int_0^t \|\nabla u(s)\|^2_{\b H} \: \d s +\mathbb{E}\int_0^t \|\Psi(u(s))\|_{L^1(\OO)} \: \d s + \mathbb{E}\int_0^t \int_{\OO}\Psi^*(\Psi'(u(s))) \: \d s
     \\
     \le \frac 12\|x\|^2_H  +  \left(\frac{\widetilde{C}_G}{2}+ \widetilde{C}_J\right)(1+  \mathbbm{1}_{L<\infty}|\OO| L^2)t +  \left( \sqrt{\widetilde C_F}+\left(\frac{\widetilde{C}_G}{2}+ \widetilde{C}_J \right)\mathbbm{1}_{L=+ \infty} \right)\mathbb{E}\int_0^t\left(1+\|u(s)\|_H\right)^2\, {\rm d}s.
   \end{multline*}
Since \luca{ $\Psi^*$ is bounded from below}, from \marghe{Lemma \ref{lem:apriori1}} \andrea{we infer the first claim}. Let us now prove \andrea{the second claim}.
 We apply the It\^o formula to the functional $\mathcal{H}_0$
evaluated at $v = u(\andrea{t})$ for any $\andrea{t} \geq 0$.
Arguing as in the \textit{Second estimate} of Section \ref{sec:unif_est}, we obtain, for any $t \ge 0$
\begin{equation*}
     \E \supp \|\Phi(u(\tau))\|_{L^1(\OO)} + \E\int_0^t \|\Phi'(u(s))\|^2_H \: \d s 
    \leq C \left[ 1 + \|\Phi(x)\|_{L^1(\OO)} +\E  \int_0^t \|\Phi(u(s))\|_{L^1(\OO)}  \: \d s \right].
\end{equation*}
From the Gronwall lemma and the first claim we thus infer 
\[
 \mathbb{E} \int_0^t \|\Phi'(u(s))\|_{H}^2\,{\rm d}s \le C(1+ \|\Phi(x)\|_{L^1(\OO)}),
\]
and \andrea{the proof is complete}.
\end{proof}
\bigskip
\textbf{Acknowldegments.} The present research has been supported by MUR, 
grant Dipartimento di Eccellenza 2023-2027.
A.D.P., L.S. and M.Z. are members of Gruppo Nazionale per l’Analisi Matematica, la Probabilità e le loro Applicazioni (GNAMPA), Istituto Nazionale di Alta Matematica (INdAM).
These authors gratefully  acknowledge the financial support of the project  ``Prot. P2022TX4FE\_02 -  Stochastic particle-based anomalous reaction-diffusion models with heterogeneous interaction for radiation therapy'' financed by the European Union - Next Generation EU, Missione 4-Componente 1-CUP: D53D23018970001. Furthermore, M.F. acknowledges the support of the state of Upper Austria and A.D.P. acknowledges support from the European Union (ERC,
NoisyFluid, No. 101053472).

\printbibliography
\end{document}